\documentclass[12pt]{article}
\usepackage{mathrsfs}
\usepackage{graphicx}
\usepackage{subfig}
\usepackage{epstopdf}
\usepackage{amsmath}
\usepackage{amsfonts}
\usepackage{amssymb, amsmath, cite}
\usepackage{cases}
\usepackage{setspace}

\newtheorem{theorem}{Theorem}
\newtheorem{lemma}{Lemma}

\newtheorem{remark}{Remark}

\newbox\qedbox
\newenvironment{proof}{\smallskip\noindent{\bf Proof.}\hskip \labelsep}%
                        {\hfill\penalty10000\copy\qedbox\par\medskip}

\newcommand{\bfR}{{\Bbb R}}
\newcommand{\bfC}{{\Bbb C}}

\newcommand{\ii}{\text{i}}
\newcommand{\e}{\text{e}}
\newcommand{\dd}{\text{d}}
\newcommand{\Om}{\Omega}

\newcommand{\nn}{\nonumber}
\newcommand\be{\begin{equation}}
\newcommand\ee{\end{equation}}
\newcommand{\bea}{\begin{eqnarray}}
\newcommand{\eea}{\end{eqnarray}}
\newcommand\berr{\begin{eqnarray*}}
\newcommand\eerr{\end{eqnarray*}}
{\setlength\arraycolsep{2pt}

\begin{document}

\title{Long-time asymptotic behavior for an extended modified Korteweg-de Vries equation}
\author{ Nan Liu$^{a,}$\footnote{Corresponding author.},\, Boling Guo$^{a}$,\, Deng-Shan Wang$^{b}$,\, Yufeng Wang$^{c}$\\
$^a${\footnotesize{\em Institute of Applied Physics and Computational Mathematics,  Beijing 100088, P.R. China}}\\
$^b${\footnotesize{\em School of Applied Science, Beijing Information Science and Technology University, Beijing 100192, P.R. China}}\\
$^c${\footnotesize{\em College of Science, Minzu University of China, Beijing 100081, P.R. China}}\\ \setcounter{footnote}{-1}\footnote{E-mail address: ln10475@163.com (N. Liu).}}

\date{}
\maketitle

\begin{quote}
{{{\bfseries Abstract.} We investigate an integrable extended modified Korteweg-de Vries equation on the line with the initial value belonging to the Schwartz space. By performing the nonlinear steepest descent analysis of an associated matrix Riemann--Hilbert problem, we obtain the explicit leading-order asymptotics of the solution of this initial value problem as time $t$ goes to infinity. For a special case $\alpha=0$, we present the asymptotic formula of the solution to the extended modified Korteweg-de Vries equation in region $\mathcal{P}=\{(x,t)\in\bfR^2|0<x\leq Mt^{\frac{1}{5}},t\geq3\}$ in terms of the solution of a fourth order Painlev\'e II equation.

}

 }
\end{quote}

\section{Introduction}
\setcounter{equation}{0}
It is a well-known fact that the modified Korteweg-de Vries (mKdV) equation is a fundamental completely integrable model in solitary waves theory, and is given in canonical form as
\be\label{1.2}
u_t+6\sigma u^2u_x+u_{xxx}=0,
\ee
where $\sigma=\pm1$, $u=u(x,t)$ is a real function with evolution variable $t$ and transverse variable $x$. This equation gives rise to multiple soliton solutions and multiple singular soliton solutions for $\sigma=+1$ and $\sigma=-1$, respectively. Moreover, the mKdV equation has significant applications in various physical contexts such as the generation of supercontinuum in optical fibres, acoustic waves in certain anharmonic lattices, nonlinear Alfv\'en waves propagating in plasma and fluid dynamics.

In this paper, we investigate an extended modified Korteweg-de Vries (emKdV) equation \cite{WX}, which takes the form
\begin{equation}\label{1.1}
u_t+\alpha(6u^2u_x+u_{xxx})+\beta(30u^4u_x+10u_x^3+40uu_xu_{xx}+10u^2u_{xxx}+u_{xxxxx})=0,
\end{equation}
where $\alpha>0$ and $\beta>0$ stand for the third- and fifth-order dispersion coefficients matching with the relevant
nonlinear terms, respectively. Moreover, \eqref{1.1} also has certain application for the description of nonlinear internal waves in a fluid stratified by both density and current \cite{GPP,PPL}. Equation \eqref{1.1} is integrable, the infinitely many conservation laws have been constructed based on the Lax pair, meanwhile, periodic and rational solutions have been also obtained by means of the $N$-fold Darboux transformation in a recent paper \cite{WX}. The Painlev\'e test and multi-soliton solutions via the simplified Hirota direct method for equation \eqref{1.1} have been recently studied in \cite{AMW}. However, it is noted that the long-time asymptotics for the emKdV equation \eqref{1.1} on the line were not analyzed to the best of our knowledge.

In particular, the purpose of present paper is to consider the initial-value problem (IVP) for the emKdV equation \eqref{1.1} on the line by a Riemann--Hilbert (RH) approach. Assuming that the initial data $u(x,0)=u_0(x)$ are smooth and decay sufficiently fast as $|x|\rightarrow\infty$, that is, $u_0(x)\in\mathcal{S}(\bfR)$,  one then can show that the solution $u(x,t)$ of the IVP for \eqref{1.1} can be represented in terms of the solution of a $2\times2$ matrix RH problem formulated in the complex $k$-plane with the jump matrices given in terms of two spectral functions $a(k)$, $b(k)$ obtained from the initial value $u_0(x)$. Then, this representation obtained allows us to apply the nonlinear steepest descent method for the associated RH problem and to obtain a detailed description for the leading term of the asymptotics of the solution for the Cauchy problem.

The nonlinear steepest descent method was first introduced in 1993 by Deift and Zhou \cite{PD}, where they derived the long-time asymptotics for the IVP for the mKdV equation \eqref{1.2} with $\sigma=-1$. It then turns out to be very successful for analyzing the long-time asymptotics of IVPs for a large range of nonlinear integrable evolution equations in a rigorous and transparent form. Numerous new significant results about the asymptotics theory of initial-value and initial-boundary value problems for different completely integrable nonlinear equations were obtained based on the analysis of the corresponding RH problems \cite{GB1,AB2,AB3,RB,XJ1,PD1,XJ,HL,ZQZ,JL1,GL1,GL2,GL3,F3,AB,JL2,DS}.

Developing and extending the methods used in \cite{PD,JL3}, our goal here is to explore the long-time asymptotics of the solution $u(x,t)$ for the emKdV equation \eqref{1.1} on the line. Compared with other integrable equations, the long-time asymptotic analysis for \eqref{1.1} presents some distinctive features. For example, the spectral curve of emKdV equation \eqref{1.1} is more involved and it  possesses four stationary points, which is different from that of mKdV equation and Hirota equation considered in \cite{PD,HL} where the phase function has only two critical points. We note that in the case of the Camassa--Holm equation \cite{ABCH,ABKST}, there is a sector $-\frac{1}{4}+C<c=\frac{x}{\aleph t}<0$ where the corresponding phase function also has four stationary points. Moreover, in the case of the Degasperis--Procesi equation \cite{AB3}, depending on the range of $x/t$, one can also have four stationary points. However, our main asymptotic analysis still presents many particular pictures different from these literatures (see Sections 3 and 4). Therefore, the study of the long-time asymptotics for the IVP for \eqref{1.1} on the line is more interesting. Our main results of this paper are summarized by the following theorems.
\begin{theorem}\label{the4.2}
Suppose that $u_0(x)$ lie in the Schwartz space $S(\bfR)$ and be such that no discrete spectrum is present. Then, for any positive constant $\varepsilon>0$, as $t\rightarrow\infty$, the solution $u(x,t)$ of the Cauchy problem for emKdV equation \eqref{1.1} on the line satisfies the following asymptotic formula
\be
u(x,t)=-\frac{u_{as}(x,t)}{\sqrt{t}}+O\bigg(\frac{\ln t}{t}\bigg),\quad t\rightarrow\infty,~\xi=\frac{x}{t}\in\bigg(-\frac{9\alpha^2}{20\beta}+\varepsilon,-\varepsilon\bigg),
\ee
where the error term is uniform with respect to $x$ in the given range, and the leading-order
coefficient $u_{as}(x,t)$ is given by
\be
\begin{aligned}
u_{as}(x,t)&=\sqrt{\frac{\nu(k_1)}{k_1(3\alpha-40\beta k_1^2)}}\cos\bigg(16tk_1^3(8\beta k_1^2-\alpha)\\
&-\nu(k_1)\ln(16t(k_2-k_1)^2(3\alpha k_1-40\beta k_1^3))+\phi_a(\xi)\bigg)\\
&+\sqrt{\frac{\nu(k_2)}{k_2(40\beta k_2^2-3\alpha)}}\cos\bigg(16tk_2^3(8\beta k_2^2-\alpha)\\
&+\nu(k_2)\ln(16t(k_2-k_1)^2(40\beta k_2^3-3\alpha k_2))+\phi_b(\xi)\bigg),
\end{aligned}
\ee
where
\bea
\phi_a(\xi)&=&-\frac{\pi}{4}-\arg r(k_1)+\arg\Gamma(\ii\nu(k_1))+2\nu(k_1)\ln\bigg(\frac{k_1+k_2}{2k_1}\bigg)\nn\\
&&-\frac{1}{\pi}\int_{k_1}^{k_2}\ln\bigg(\frac{1+|r(s)|^2}{1+|r(k_1)|^2}\bigg)
\bigg(\frac{1}{s-k_1}-\frac{1}{s+k_1}\bigg)\dd s,\nn\\
\phi_b(\xi)&=&\frac{\pi}{4}-\arg r(k_2)-\arg\Gamma(\ii\nu(k_2))+2\nu(k_2)\ln\bigg(\frac{2k_2}{k_1+k_2}\bigg)\nn\\
&&-\frac{1}{\pi}\int_{k_1}^{k_2}\ln\bigg(\frac{1+|r(s)|^2}{1+|r(k_2)|^2}\bigg)\bigg(\frac{1}{s-k_2}-\frac{1}{s+k_2}\bigg)\dd s,\nn
\eea
and $k_1,~k_2$, $\nu(k_1),~\nu(k_2)$ are defined by \eqref{3.6}, \eqref{3.7}, \eqref{3.43} and \eqref{3.45}, respectively.
\end{theorem}
\begin{remark}
For $\xi<-\frac{9\alpha^2}{20\beta}$, there are no real critical points for the phase function $\Phi(k)$. Thus, it is easy to proof that the solution $u(x,t)$ of emKdV equation \eqref{1.1} is rapidly decreasing as $t\rightarrow\infty$. However, for $\xi>0$, there are two different real stationary points $\pm k_0=\pm\sqrt{\frac{3\alpha}{40\beta}\bigg(1+\sqrt{1+\frac{20\beta\xi}{9\alpha^2}}\bigg)}$, this implies that it is possible to deform the RH problem through a series of transformations in exactly the same way as in the similarity region for the mKdV equation to find the asymptotics (one also can follow the strategy used in Section 3).
\end{remark}
\begin{theorem}\label{the4.3}
Under the assumptions of Theorem \ref{the4.2},  the solution $u(x,t)$ of equation \eqref{4.1}, i.e., $\alpha=0$ in emKdV equation \eqref{1.1}, satisfies the following asymptotic formula as $t\rightarrow\infty$:
\be
u(x,t)=\bigg(\frac{8}{5\beta t}\bigg)^{\frac{1}{5}}u_p\bigg(\frac{-x}{(20\beta t)^{\frac{1}{5}}}\bigg)+O(t^{-\frac{2}{5}}),\quad 0<x\leq Mt^{\frac{1}{5}},
\ee
where the formula holds uniformly with respect to $x$ in the given range for any fixed $M>1$ and the function $u_p(y)$ denotes the solution of the fourth order Painlev\'e II equation (A.5).
\end{theorem}

The organization of this paper is as follows. In Section 2, we show how the solution of emKdV equation \eqref{1.1} can be expressed in terms of the solution of a $2\times2$ matrix RH problem and give an auxiliary theorem which is useful for determining the long-time asymptotics. In Section 3, we derive the long-time asymptotic behavior of the solution of the emKdV equation \eqref{1.1} to prove our first main Theorem \ref{the4.2} in physically interesting region. In Section 4, we present the asymptotic formula of the solution to a particular case $\alpha=0$ of emKdV equation \eqref{1.1} in region $0<x\leq Mt^{\frac{1}{5}}$. A few facts related to the RH problem associated with the fourth order Painlev\'e II equation are collected in Appendix.
\section{Preliminaries}
\setcounter{equation}{0}
\setcounter{lemma}{0}
\setcounter{theorem}{0}
\subsection{Riemann--Hilbert formalism}
The Lax pair of equation \eqref{1.1} is \cite{WX}
\begin{equation}\label{2.1}
\begin{aligned}
\Psi_x&=X\Psi,\quad X=\ii k\sigma_3+U,\\
\Psi_t&=T\Psi,~\quad T=(-16\ii\beta k^5+4\ii\alpha k^3)\sigma_3+V,
\end{aligned}
\end{equation}
(namely, equation \eqref{1.1} is the compatibility condition $X_t-T_x+[X,T]=0$ of equation \eqref{2.1}), where $\Psi(x,t;k)$ is a $2\times2$ matrix-valued function, $k\in\bfC$ is the spectral parameter and
\begin{eqnarray}
&&\sigma_3={\left( \begin{array}{cc}
1 & 0 \\[4pt]
0 & -1\\
\end{array}
\right )},\quad
U={\left( \begin{array}{cc}
0 & u \\[4pt]
-u & 0 \\
\end{array}
\right )},\quad V={\left( \begin{array}{cc}
A & B \\[4pt]
C & -A \\
\end{array}
\right )},\label{2.2}\\
&&A=8\ii\beta k^3u^2-\ii(6\beta u^4+2\alpha u^2+4\beta uu_{xx}-2\beta u_x^2)k,\nn\\
&&B=-16\beta k^4u+8\ii\beta k^3u_x+(8\beta u^3+4\alpha u+4\beta u_{xx})k^2-\ii(12\beta u^2u_x+2\beta u_{xxx}+2\alpha u_x)k\nn\\
&&\qquad-6\beta u^5-2\alpha u^3-10\beta u^2u_{xx}-10\beta uu_x^2-\beta u_{xxxx}-\alpha u_{xx},\nn\\
&&C=-B+16\ii\beta k^3u_x-\ii (24\beta u^2u_x+4\alpha u_x+4\beta u_{xxx})k.\nn
\end{eqnarray}
Introducing a new eigenfunction $\mu(x,t;k)$ by
\begin{equation}\label{2.3}
\mu(x,t;k)=\Psi(x,t;k) \e^{-\ii[k x+(-16\beta k^5+4\alpha k^3)t]\sigma_3},
\end{equation}
we obtain the equivalent Lax pair
\begin{equation}\label{2.4}
\begin{aligned}
&\mu_x-\ii k[\sigma_3,\mu]=U\mu,\\
&\mu_t+\ii(16\beta k^5-4\alpha k^3)[\sigma_3,\mu]=V\mu.\\
\end{aligned}
\end{equation}
We now consider the spectral analysis of the $x$-part of \eqref{2.4}. Define two solutions $\mu_1$ and $\mu_2$ of the $x$-part of \eqref{2.4} by the following Volterra integral equations
\bea
\mu_1(x,t;k)&=&I+\int_{-\infty}^x\e^{\ii k(x-x')\hat{\sigma}_3}[U(x',t)\mu_1(x',t;k)]\dd x',\label{2.5}\\
\mu_2(x,t;k)&=&I-\int^{\infty}_x\e^{\ii k(x-x')\hat{\sigma}_3}[U(x',t)\mu_2(x',t;k)]\dd x',\label{2.6}
\eea
where $\hat{\sigma}_3$ acts on a $2\times2$ matrix $X$ by $\hat{\sigma}_3X=[\sigma_3,X]$, and $\e^{\hat{\sigma}_3}=\e^{\sigma_3}X\e^{-\sigma_3}$. We denote by $\mu^{(1)}$ and $\mu^{(2)}$ the columns of a $2\times2$ matrix $\mu=(\mu^{(1)}~\mu^{(2)})$. Then it follows from \eqref{2.5}-\eqref{2.6} that for all $(x,t)$:

(i) $\det\mu_j=1$, $j=1,2$.

(ii) $\mu^{(1)}_2$ and $\mu^{(2)}_1$ are analytic and bounded in $\{k\in\bfC|\mbox{Im}k>0\}$, and $(\mu^{(1)}_2~\mu^{(2)}_1)\rightarrow I$ as $k\rightarrow\infty.$

(iii) $\mu^{(1)}_1$ and $\mu^{(2)}_2$ are analytic and bounded in $\{k\in\bfC|\mbox{Im}k<0\}$, and $(\mu^{(1)}_1~\mu^{(2)}_2)\rightarrow I$ as $k\rightarrow\infty.$

(iv) $\{\mu_j\}_1^2$ are continuous up to the real axis.

(v) Symmetry:
\be\label{2.7}
\overline{\mu_j(x,t;\bar{k})}=\mu_j(x,t;-k)=\sigma_2\mu_j(x,t;k)\sigma_2,
\ee
where $\sigma_2$ is the second Pauli matrix,
\berr
\sigma_2=\begin{pmatrix}
0 & -\ii\\
\ii & 0\\
\end{pmatrix}.
\eerr
The symmetry relation \eqref{2.7} can be proved easily due to the symmetries of the matrix $X$:
$$\overline{X(x,t;\bar{k})}=X(x,t;-k)=\sigma_2X(x,t;k)\sigma_2.$$

The solutions of the system of differential equation \eqref{2.4} must be related by a matrix independent of $x$ and $t$, therefore,
\bea\label{2.8}
&\mu_1(x,t;k)=\mu_2(x,t;k)\e^{\ii[kx+(-16\beta k^5+4\alpha k^3)t]\hat{\sigma}_3}s(k),~\det s(k)=1,~k\in\bfR.
\eea
Evaluation at $x\rightarrow\infty,t=0$ gives
\be\label{2.9}
s(k)=\lim_{x\rightarrow\infty}\e^{-\ii kx\hat{\sigma}_3}\mu_1(x,0;k),
\ee
that is,
\be
s(k)=I+\int_{-\infty}^\infty\e^{-\ii kx\hat{\sigma}_3}[U(x,0)\mu_1(x,0;k)]\dd x.\label{2.10}
\ee
Due to the symmetry \eqref{2.7}, the matrix-valued spectral function $s(k)$ can be  defined in terms of two scalar spectral functions $a(k)$ and $b(k)$ by
\be\label{2.11}
s(k)=\begin{pmatrix}
\bar{a}(k) & b(k)\\[4pt]
-\bar{b}(k) & a(k)\\
\end{pmatrix},
\ee
where $\bar{a}(k)=\overline{a(\bar{k})}$ and $\bar{b}(k)=\overline{b(\bar{k})}$ indicate the Schwartz conjugates.
The spectral functions $a(k)$ and $b(k)$ can be determined by $u_0(x)$ through the solution of equation \eqref{2.10}.
On the other hand, $a(k)$ is analytic in the half-plane $\{k\in\bfC|\mbox{Im}k>0\}$ and continuous in $\{k\in\bfC|\mbox{Im}k\geq0\}$, and $a(k)\rightarrow1$ as $k\rightarrow\infty$.
Furthermore, $|a(k)|^2+|b(k)|^2=1$ for $k\in\bfR$. Finally, $a(-k)=\bar{a}(k),~b(-k)=\bar{b}(k)$.

Assuming $u(x,t)$ be a solution of equation \eqref{1.1}, the analytic properties of $\mu_j(x,t;k)$ stated above allow us to define a piecewise meromorphic, $2\times2$ matrix-valued function $M(x,t;k)$ by
\be\label{2.12}
M(x,t;k)=\left\{
\begin{aligned}
&\bigg(\frac{\mu_2^{(1)}(x,t;k)}{a(k)}~\mu_1^{(2)}(x,t;k)\bigg),\quad \text{Im}k>0,\\
&\bigg(\mu_1^{(1)}(x,t;k)~\frac{\mu_2^{(2)}(x,t;k)}{\bar{a}(k)}\bigg),\quad \text{Im}k<0.
\end{aligned}
\right.
\ee
Then, for each $x\in\bfR$ and $t\geq0$, the boundary values $M_{\pm}(x,t;k)$ of $M$ as $k$ approaches $\bfR$ from the sides $\pm\text{Im}k>0$ are related as follows:
\be\label{2.13}
M_+(x,t;k)=M_-(x,t;k)J(x,t;k),\quad k\in\bfR,
\ee
with
\be\label{2.14}
\begin{aligned}
&J(x,t;k)=\begin{pmatrix}
1+|r(k)|^2 ~& \bar{r}(k)\e^{-t\Phi(k)}\\[4pt]
r(k)\e^{t\Phi(k)} ~& 1
\end{pmatrix},\\
&r(k)=\frac{\bar{b}(k)}{a(k)},\quad \Phi(k)=2\ii(-k\frac{x}{t}+16\beta k^5-4\alpha k^3).
\end{aligned}
\ee
In view of the properties of $\mu_j(x,t;k)$ and $s(k)$, $M(x,t;k)$ also satisfies the following properties:\\
(i) Behavior at $k=\infty$:
\be\label{2.15}
M(x,t;k)\rightarrow I~\text{as}~k\rightarrow\infty.
\ee
(ii) Symmetry:
\be\label{2.16}
\overline{M(x,t;\bar{k})}=M(x,t;-k)=\sigma_2M(x,t;k)\sigma_2.
\ee
(iii) Residue conditions: Let $\{k_j\}_1^N$ be the set of zeros of $a(k)$. We assume these zeros are finite in number, simple  and no zero is real, then $M(x,t;k)$ satisfies the following residue conditions:
\be\label{2.17}
\begin{aligned}
&\text{Res}_{k=k_j}M^{(1)}(x,t;k)=\frac{\e^{t\Phi(k_j)}}{\dot{a}(k_j)b(k_j)}M^{(2)}(x,t;k_j)
=\ii\chi_j\e^{t\Phi(k_j)}M^{(2)}(x,t;k_j),\\
&\text{Res}_{k=\bar{k}_j}M^{(2)}(x,t;k)=-\frac{\e^{-t\Phi(\bar{k}_j)}}
{\overline{\dot{a}(k_j)}\overline{b(k_j)}}M^{(1)}(x,t;\bar{k}_j)
=\ii\bar{\chi}_j\e^{-t\Phi(\bar{k}_j)}M^{(1)}(x,t;\bar{k}_j).
\end{aligned}
\ee
\begin{theorem}\label{th1}
Let $\{r(k),\{k_j,\chi_j\}_1^N\}$ be the spectral data determined by $u_0(x)$, and define $M(x,t;k)$ as the solution of the associated RH \eqref{2.13} with the jump matrix \eqref{2.14}, the normalization condition \eqref{2.15} and the residue conditions \eqref{2.17}. Then, $M(x,t;k)$ exists and is unique. Define $u(x,t)$ in terms of $M(x,t;k)$ by
\bea\label{2.18}
u(x,t)=-2\ii\lim_{k\rightarrow\infty}(kM(x,t;k))_{12}.
\eea
Then $u(x,t)$ solves the emKdV equation \eqref{1.1}. Furthermore, $u(x,0)=u_0(x)$.
\end{theorem}
\begin{proof}
In the case when $a(k)$ has no zeros, the existence and uniqueness for the solution of above RH problem is a consequence of a `vanishing lemma' for the associated RH problem with the vanishing condition at infinity $M(k)=O(1/k)$, $k\rightarrow\infty$ (see \cite{LF} since $J^\dag(\bar{k})=J(k)$). If $a(k)$ has zeros, the singular RH problem can be mapped to a regular one following the approach of \cite{F4}. Moreover, it follows from standard arguments using the dressing method \cite{Fokas} that if $M$ solves the above RH problem and $u(x,t)$ is defined by \eqref{2.18}, then $u(x,t)$ solves the emKdV equation \eqref{1.1}. One observes that for $t=0$, the RH problem reduces to that associated with $u_0(x)$, which yields $u(x,0)=u_{0}(x)$, owing to the uniqueness of the solution of the RH problem.
\end{proof}
\subsection{A model RH problem}
After the formulation of the main RH problem, the main idea of analysis of the long-time behavior is to reduce the original RH problem to a model RH problem which can be solved exactly. The following theorem is turned out suitable for determining asymptotics of a class of RH problems which arise in the study of long-time asymptotics.
\begin{figure}[htbp]
  \centering
  \includegraphics[width=3in]{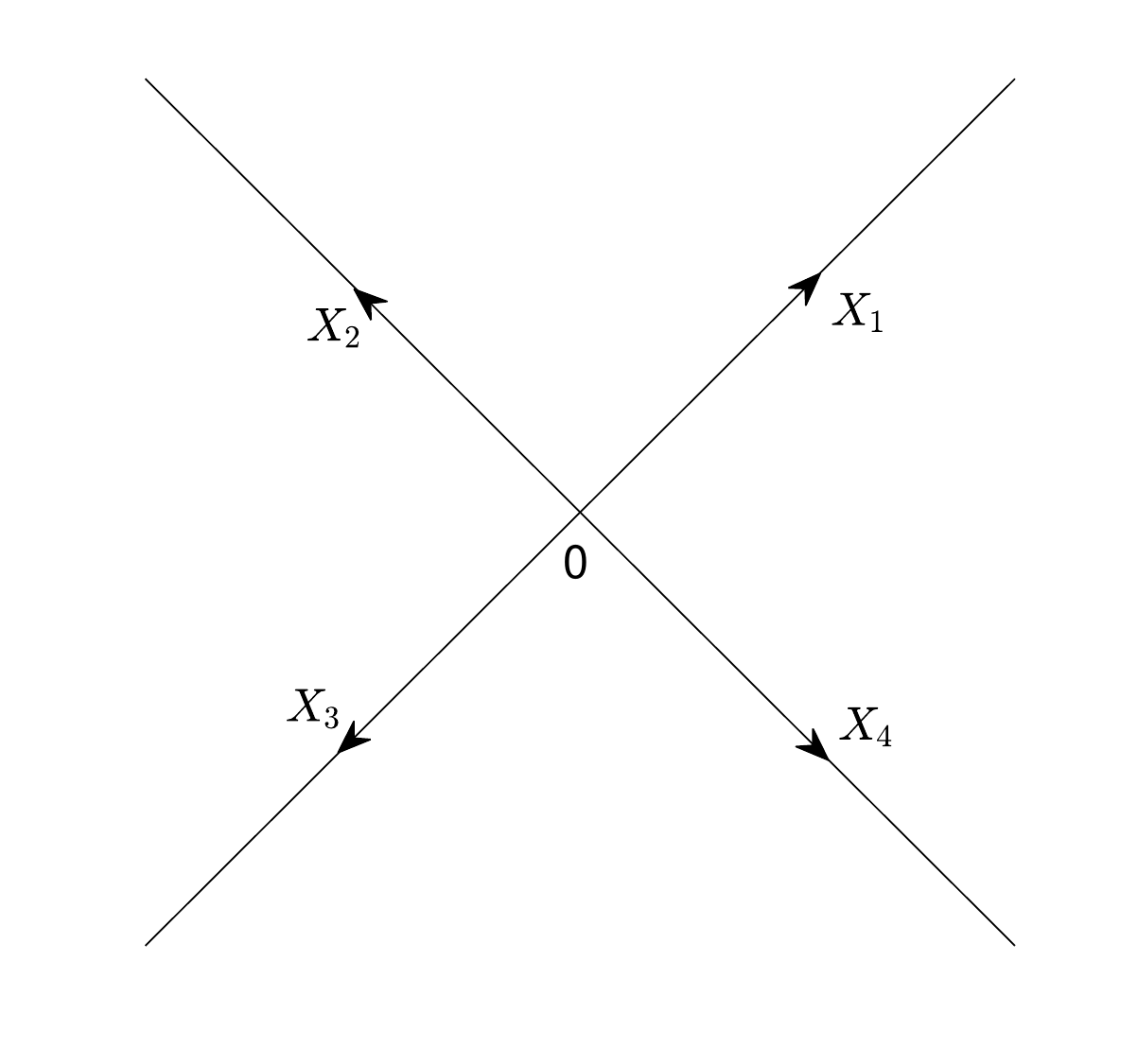}
  \caption{The contour $X=X_1\cup X_2\cup X_3\cup X_4$.}\label{fig1}
\end{figure}

Let $X=X_1\cup X_2\cup X_3\cup X_4\subset\bfC$ be the cross defined by
\be\label{2.19}
\begin{aligned}
X_1&=\{l\e^{\frac{\ii\pi}{4}}|0\leq l<\infty\},~~~ X_2=\{l\e^{\frac{3\ii\pi}{4}}|0\leq l<\infty\},\\
X_3&=\{l\e^{-\frac{3\ii\pi}{4}}|0\leq l<\infty\},~ X_4=\{l\e^{-\frac{\ii\pi}{4}}|0\leq l<\infty\},
\end{aligned}
\ee
and oriented as in Fig. \ref{fig1}. Define the function $\nu:\bfC\rightarrow(0,\infty)$ by $\nu(q)=\frac{1}{2\pi}\ln(1+|q|^2)$. We consider the following RH problems parametrized by $q\in\bfC$:
\be\label{2.20}\left\{
\begin{aligned}
&M^X_+(q,z)=M^X_-(q,z)J^X(q,z),~~z\in X,\\
&M^X(q,z)\rightarrow I,~~~\qquad\qquad\qquad~z\rightarrow\infty,
\end{aligned}
\right.
\ee
where the jump matrix $J^X(q,z)$ is defined by
\be\label{2.21}
J^X(q,z)=\left\{
\begin{aligned}
&\begin{pmatrix}
1 ~& 0\\[4pt]
q\e^{\frac{\ii z^2}{2}}z^{2\ii\nu(q)} ~& 1
\end{pmatrix},~~~\quad\qquad z\in X_1,\\
&\begin{pmatrix}
1 ~& -\frac{\bar{q}}{1+|q|^2}\e^{-\frac{\ii z^2}{2}}z^{-2\ii\nu(q)}\\[4pt]
0 ~& 1
\end{pmatrix},~~ z\in X_2,\\
&\begin{pmatrix}
1 ~& 0\\[4pt]
-\frac{q}{1+|q|^2}\e^{\frac{\ii z^2}{2}}z^{2\ii\nu(q)} ~& 1
\end{pmatrix},~~\quad z\in X_3,\\
&\begin{pmatrix}
1 ~& \bar{q}\e^{-\frac{\ii z^2}{2}}z^{-2\ii\nu(q)}\\[4pt]
0 ~& 1
\end{pmatrix},~~~\qquad~ z\in X_4.
\end{aligned}
\right.
\ee
Then we have the following theorem.
\begin{theorem}\label{th2}
The RH problem \eqref{2.20} has a unique solution $M^X(q,z)$ for each $q\in \bfC$. This solution satisfies
\be\label{2.22}
M^X(q,z)=I-\frac{\ii}{z}\begin{pmatrix}
0 ~& \beta^X(q)\\[4pt]
\overline{\beta^X(q)} ~& 0
\end{pmatrix}+O\bigg(\frac{q}{z^2}\bigg),\quad z\rightarrow\infty,~q\in\bfC,
\ee
where the error term is uniform with respect to $\arg z\in[0,2\pi]$ and the function $\beta^X(q)$ is given by
\be\label{2.23}
\beta^X(q)=\sqrt{\nu(q)}\e^{\ii\big(\frac{\pi}{4}-\arg q-\arg\Gamma(\ii\nu(q))\big)},\quad q\in\bfC,
\ee
where $\Gamma(\cdot)$ denotes the standard Gamma function. Moreover, for each compact subset $\mathcal{D}$ of $\bfC$,
\be\label{2.24}
\sup_{q\in\mathcal{D}}\sup_{z\in\bfC\setminus X}|M^X(q,z)|<\infty
\ee
and
\be\label{2.25}
\sup_{q\in\mathcal{D}}\sup_{z\in\bfC\setminus X}\frac{|M^X(q,z)-I|}{|q|}<\infty.
\ee
\end{theorem}
\begin{proof}
For the proof of this theorem, we refer the readers to see \cite{PD,PD1,JL2}.
\end{proof}
\section{Long-time asymptotics }
\setcounter{equation}{0}
\setcounter{lemma}{0}
\setcounter{theorem}{0}
In this section, we aim to transform the associated original RH problem \eqref{2.15} to a solvable RH problem and then find the explicitly asymptotic formula for the emKdV equation \eqref{1.1}. In the following analysis, we suppose that $a(k)\neq0$ for $\{k\in\bfC|\text{Im}k\geq0\}$ so that no discrete spectrum is present. Namely, we consider the following RH problem
\be\label{3.1}
\left\{
\begin{aligned}
&M_+(x,t;k)=M_-(x,t;k)J(x,t;k),~k\in \bfR,\\
&M(x,t;k)\rightarrow I,~\qquad\qquad\qquad\qquad~k\rightarrow\infty,
\end{aligned}
\right.
\ee
where the jump matrix $J(x,t;k)$ is defined by
\be\label{3.2}
\begin{aligned}
&J(x,t;k)=\begin{pmatrix}
1+|r(k)|^2 ~& \bar{r}(k)\e^{-t\Phi(k)}\\[4pt]
r(k)\e^{t\Phi(k)} ~& 1
\end{pmatrix},\\
&r(k)=\frac{\bar{b}(k)}{a(k)},~ \Phi(k)=2\ii(-k\xi+16\beta k^5-4\alpha k^3),~ \xi=\frac{x}{t}.
\end{aligned}
\ee
In view of the symmetry relation in \eqref{2.7}, we conclude that
\be\label{3.3}
r(-k)=\overline{r(\bar{k})},\quad k\in\bfR.
\ee
Moreover, the relation between the solution $u(x,t)$ of the emKdV equation \eqref{1.1} and $M(x,t;k)$ is
\be\label{3.4}
u(x,t)=-2\ii\lim_{k\rightarrow\infty}(kM(x,t;k))_{12}.
\ee

The jump matrix $J$ defined in \eqref{3.2} involves the exponentials $\e^{\pm t\Phi}$, therefore,  the sign structure of the quantity Re$\Phi(k)$ plays an important role in the following analysis. In particular, we suppose \be\label{3.5}
-\frac{9\alpha^2}{20\beta}<\xi<0.
\ee
It follows that there are four different real stationary points located at the points where $\frac{\partial\Phi}{\partial k}=0$,
namely, at
\bea
\pm k_1&=&\pm\sqrt{\frac{3\alpha}{40\beta}\bigg(1-\sqrt{1+\frac{20\beta\xi}{9\alpha^2}}\bigg)},\label{3.6}\\
\pm k_2&=&\pm\sqrt{\frac{3\alpha}{40\beta}\bigg(1+\sqrt{1+\frac{20\beta\xi}{9\alpha^2}}\bigg)}\label{3.7}.
\eea
The signature table for Re$\Phi(k)$ is shown in Fig. \ref{fig2}.
\begin{figure}[htbp]
  \centering
  \includegraphics[width=4in]{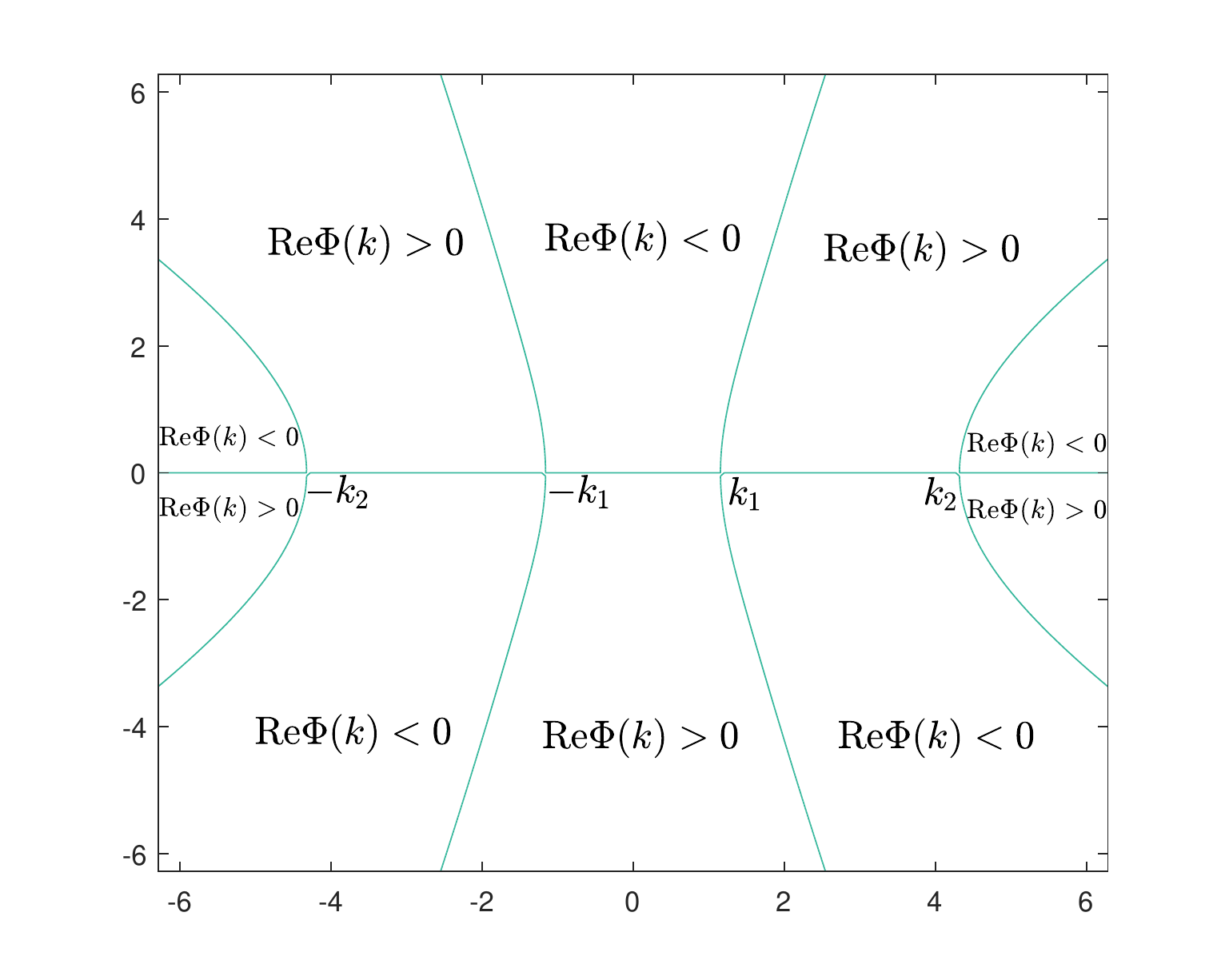}
  \caption{The signature table for Re$\Phi(k)$ in the complex $k$-plane.}\label{fig2}
\end{figure}

 Let $\varepsilon>0$ be given constant. We restrict our attention here to the physically interesting region $\xi\in\mathcal{I}=\big(-\frac{9\alpha^2}{20\beta}+\varepsilon,-\varepsilon\big)$.
\subsection{Transformations of the RH problem}
One goes from the original RH problem \eqref{3.1} for $M$ to the equivalent RH problem for the new function $M^{(1)}$ defined by
\be\label{3.8}
M^{(1)}(x,t;k)=M(x,t;k)\delta^{-\sigma_3}(k),
\ee
where the complex-valued function $\delta(k)$ is given by
\be\label{3.9}
\delta(k)=\exp\bigg\{\frac{1}{2\pi\ii}\bigg(\int_{-k_2}^{-k_1}+\int^{k_2}_{k_1}\bigg)\frac{\ln(1+|r(s)|^2)}{s-k}\dd s\bigg\},~ k\in\bfC\setminus([-k_2,-k_1]\cup[k_1,k_2]).
\ee
\begin{lemma}\label{lem1}
The function $\delta(k)$ has the following properties:

(i) $\delta(k)$ satisfies the following jump condition across the real axis oriented from $-\infty$ to $\infty$:
\berr
\delta_+(k)=\delta_-(k)(1+|r(k)|^2),~~k\in(-k_2,-k_1)\cup(k_1,k_2).
\eerr

(ii) As $k\rightarrow\infty$, $\delta(k)$ satisfies the asymptotic formula
\be\label{3.10}
\delta(k)=1+O(k^{-1}),\quad k\rightarrow\infty.
\ee

(iii) $\delta(k)$ and $\delta^{-1}(k)$ are bounded and analytic functions of $k\in\bfC\setminus([-k_2,-k_1]\cup[k_1,k_2])$ with continuous boundary values on $(-k_2,-k_1)\cup(k_1,k_2)$.

(iv) $\delta(k)$ obeys the symmetry $$\delta(k)=\overline{\delta(\bar{k})}^{-1},\quad k\in\bfC\setminus([-k_2,-k_1]\cup[k_1,k_2]).$$
\end{lemma}
Then $M^{(1)}(x,t;k)$ satisfies the following RH problem
\be\label{3.11}
M^{(1)}_+(x,t;k)=M^{(1)}_-(x,t;k)J^{(1)}(x,t;k),\quad k\in\bfR,
\ee
with the jump matrix $J^{(1)}=\delta_-^{\sigma_3}J\delta_+^{-\sigma_3}$, namely,
\bea
J^{(1)}(x,t;k)=\left\{
\begin{aligned}
&\begin{pmatrix}
1 ~& r_4(k)\delta^{2}(k)\e^{-t\Phi(k)} \\[4pt]
0 ~& 1 \\
\end{pmatrix}\begin{pmatrix}
1 ~& 0 \\[4pt]
r_1(k)\delta^{-2}(k)\e^{t\Phi(k)} ~& 1 \\
\end{pmatrix},\quad |k|>k_2,~|k|<k_1,\\
&\begin{pmatrix}
1 ~& 0 \\[4pt]
r_3(k)\delta_-^{-2}(k)\e^{t\Phi}(k) ~& 1 \\
\end{pmatrix}\begin{pmatrix}
1 ~& r_2(k)\delta_+^{2}(k)\e^{-t\Phi(k)} \\[4pt]
0 ~& 1 \\
\end{pmatrix},\quad k_1<|k|<k_2,
\end{aligned}
\right.
\eea
where we define $\{r_j(k)\}_1^4$ by
\be\label{3.13}
\begin{aligned}
r_1(k)&=r(k),\qquad\qquad~~ r_2(k)=\frac{\bar{r}(k)}{1+r(k)\bar{r}(k)},\\
r_3(k)&=\frac{r(k)}{1+r(k)\bar{r}(k)},\quad r_4(k)=\bar{r}(k).
\end{aligned}
\ee
\begin{figure}[htbp]
  \centering
  \includegraphics[width=3.5in]{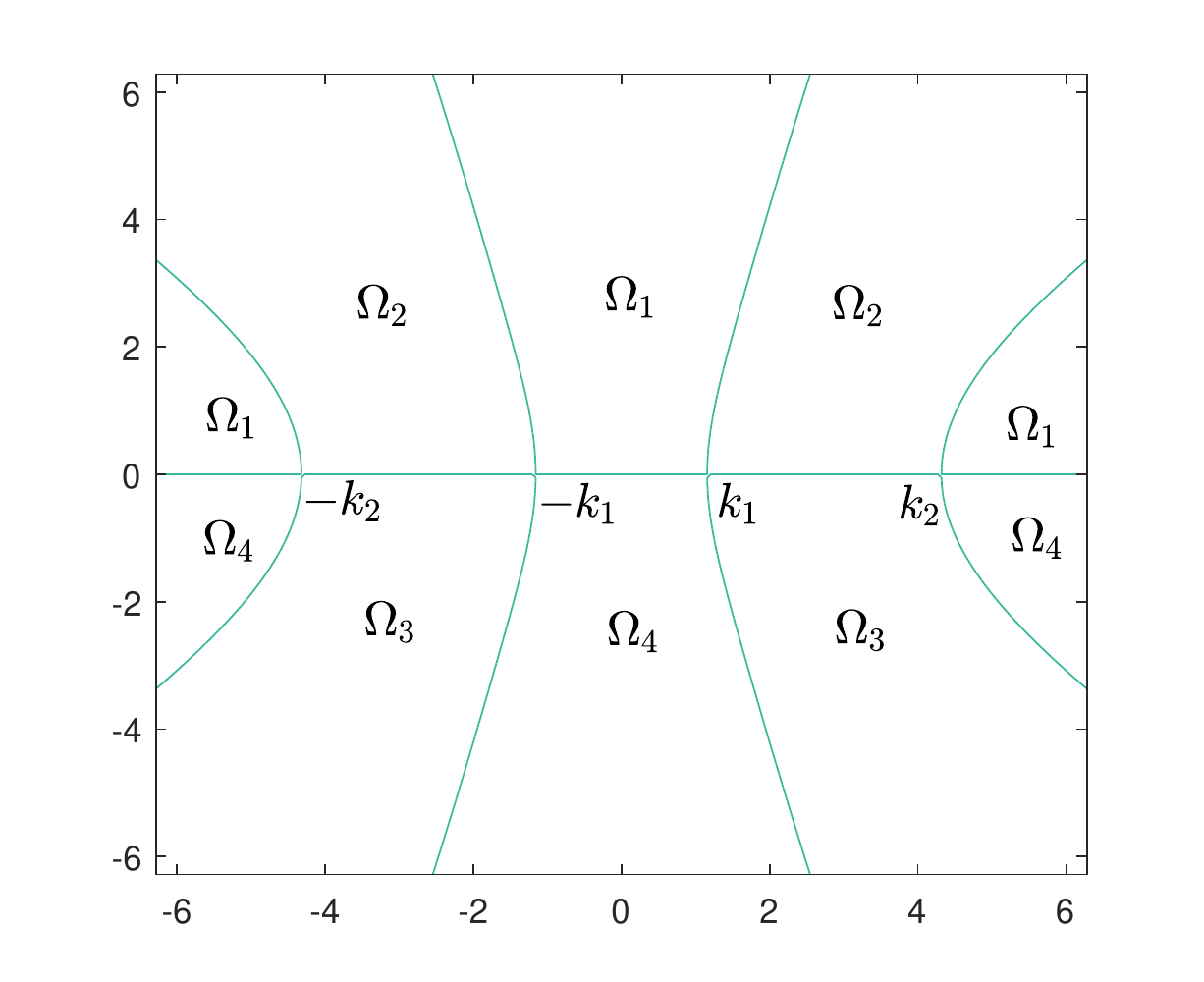}
  \caption{The open sets $\{\Omega_j\}_1^{4}$ in the complex $k$-plane.}\label{fig3}
\end{figure}

Before processing the next deformation, we first introduce analytic approximations of $\{r_j(k)\}_1^4$ following the idea of \cite{JL3}. We define the open subsets $\{\Omega_j\}_1^{4}$, as displayed in Fig. \ref{fig3} such that
\bea
\Omega_1\cup\Omega_3&=&\{k\in\bfC|\text{Re}\Phi(k)<0\},\nn\\
\Omega_2\cup\Omega_4&=&\{k\in\bfC|\text{Re}\Phi(k)>0\}.\nn
\eea
\begin{lemma}\label{lem2}
There exist decompositions
\be\label{3.14}
r_j(k)=\left\{
\begin{aligned}
&r_{j,a}(x,t,k)+r_{j,r}(x,t,k),\quad |k|>k_2,~|k|<k_1,~k\in\bfR,~j=1,4,\\
&r_{j,a}(x,t,k)+r_{j,r}(x,t,k),\quad k_1<|k|<k_2,~k\in\bfR,~j=2,3,
\end{aligned}
\right.
\ee
where the functions $\{r_{j,a}, r_{j,r}\}^4_1$ have the following properties:

(1) For $\xi\in\mathcal{I}$ and each $t>0$, $r_{j,a}(x,t,k)$ is defined and continuous for $k\in\bar{\Omega}_j$ and
analytic for $\Omega_j$, $j= 1,2,3,4$.

(2) The functions $r_{1,a}$ and $r_{4,a}$ satisfy, for~$\xi\in\mathcal{I},~t>0$,
\be\label{3.15}
|r_{j,a}(x,t,k)|\leq \frac{C}{1+|k|^2}\e^{\frac{t}{4}|\text{Re}\Phi(k)|},~k\in\bar{\Omega}_j\cap\{k\in\bfC||Rek|>k_2\},~j=1,4,
\ee
where the constant $C$ is independent of $\xi, k, t$.

(3) The $L^1, L^2$ and $L^\infty$ norms of the functions $r_{1,r}(x,t,\cdot)$ and $r_{4,r}(x,t,\cdot)$ on $(-\infty,-k_2)\cup(k_2,\infty)\cup(-k_1,k_1)$ are $O(t^{-3/2})$ as $t\rightarrow\infty$ uniformly with respect to $\xi\in\mathcal{I}$.

(4) The $L^1, L^2$ and $L^\infty$ norms of the functions $r_{2,r}(x,t,\cdot)$ and $r_{3,r}(x,t,\cdot)$ on $(-k_2,-k_1)\cup(k_1,k_2)$ are $O(t^{-3/2})$ as $t\rightarrow\infty$ uniformly with respect to $\xi\in\mathcal{I}$.

(5) For $j=1,2,3,4$, the following symmetries hold:
\be\label{3.16}
r_{j,a}(x,t,k)=\overline{r_{j,a}(x,t,-\bar{k})},~r_{j,r}(x,t,k)=\overline{r_{j,r}(x,t,-\bar{k})}.
\ee
\end{lemma}
\begin{proof}
We first consider the decomposition of $r_1(k)$. Denote $\Om_1=\Om_1^1\cup\Om_1^2\cup\Om_1^3$, where $\Om_1^1$, $\Om_1^2$ and $\Om_1^3$ denote the parts of $\Om_1$ in $\{k\in\bfC|\text{Re}k>k_2\}$, $\{k\in\bfC|\text{Re}k<-k_2\}$ and the remaining part, respectively. We first derive a decomposition of $r_1(k)$ in $\Om_1^1$, and then extend it to $\Om_1^2$ by symmetry. Then, we derive a decomposition of $r_1(k)$ in $\Om_1^3$. Since $u_0(x)\in\mathcal{S}(\bfR)$, this implies that $r_1(k)=r(k)\in \mathcal{S}(\bfR)$. Then for $n=0,1,2$, we have
\bea
r_1^{(n)}(k)&=&\frac{\dd^n}{\dd k^n}\bigg(\sum_{j=0}^6\frac{r_1^{(j)}(k_2)}{j!}(k-k_2)^j\bigg)+O((k-k_2)^{7-n}),~~ k\rightarrow k_2,\label{3.17}
\eea
Let
\be\label{3.18}
f_0(k)=\sum_{j=5}^{11}\frac{a_j}{(k-\ii)^j},
\ee
where $\{a_j\}_5^{11}$ are complex constants such that
\be\label{3.19}
f_0(k)=\sum_{j=0}^6\frac{r_1^{(j)}(k_2)}{j!}(k-k_2)^j+O((k-k_2)^7),~~k\rightarrow k_2.
\ee
It is easy to verify that \eqref{3.19} imposes seven linearly independent conditions on the $a_j$, hence the coefficients $a_j$ exist and are unique. Letting $f=r_1-f_0$, it follows that \\
(i) $f_0(k)$ is a rational function of $k\in\bfC$ with no poles in $\Om_1^1$;\\
(ii) $f_0(k)$ coincides with $r_1(k)$ to six order at $k_2$, more precisely,
\be\label{3.20}
\frac{\dd^n}{\dd k^n}f(k)=\left\{
\begin{aligned}
&O((k-k_2)^{7-n}),~ k\rightarrow k_2,\\
&O(k^{-5-n}),~~~~~\quad k\rightarrow\infty,
\end{aligned}
\quad k\in\bfR,~n=0,1,2.
\right.
\ee
\begin{figure}[htbp]
  \centering
  \includegraphics[width=4in]{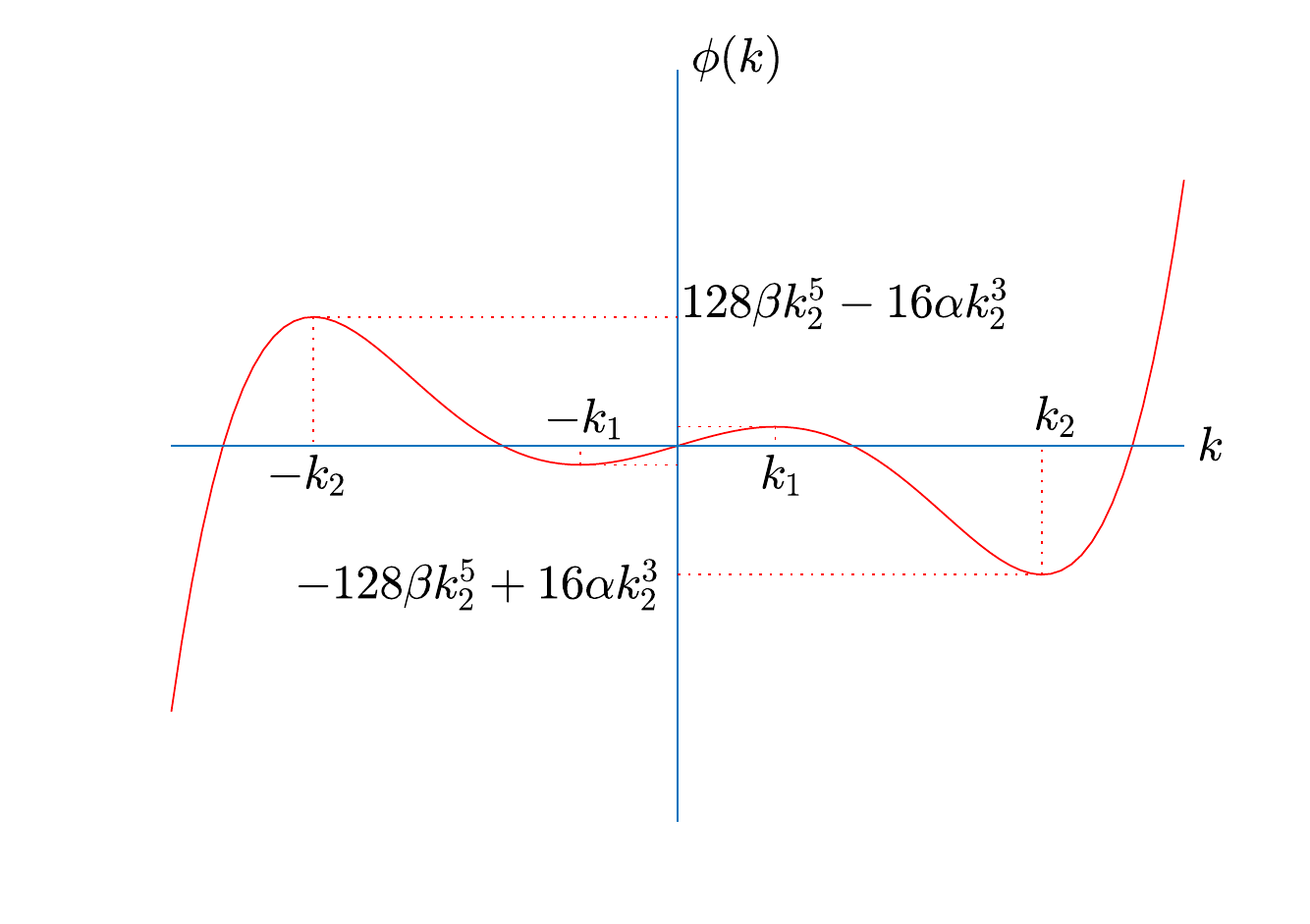}
  \caption{The graph of the function $\phi(k)$ defined in \eqref{3.21}.}\label{fig4}
\end{figure}
The decomposition of $r_1(k)$ can be derived as follows. The map $k\mapsto\phi=\phi(k)$ defined by
\be\label{3.21}
\phi(k)=-\ii\Phi(k)=2(16\beta k^5-4\alpha k^3-\xi k)
\ee
is a bijection $(k_2,\infty)\mapsto(-128\beta k_2^5+16\alpha k_2^3,\infty)$ (see Fig. \ref{fig4}), so we may define a function $F$ by
\be\label{3.22}
F(\phi)=\left\{
\begin{aligned}
&\frac{(k-\ii)^3}{k-k_2}f(k),~~\phi>-128\beta k_2^5+16\alpha k_2^3,\\
&0,\qquad\qquad~~~~~ \phi\leq-128\beta k_2^5+16\alpha k_2^3.
\end{aligned}
\right.
\ee
Then,
\berr
F^{(n)}(\phi)=\bigg(\frac{1}{160\beta(k^2-k_2^2)(k^2+k_2^2-\frac{3\alpha}{20\beta})}\frac{\partial}{\partial k}\bigg)^n\bigg(\frac{(k-\ii)^3}{k-k_2}f(k)\bigg),~~\phi>-128\beta k_2^5+16\alpha k_2^3.
\eerr
By \eqref{3.20}, $F^{(n)}(\phi)=O(|\phi|^{-3/5})$ as $|\phi|\rightarrow\infty$ for $n=0,1,2$. In particular,
\be\label{3.23}
\bigg\|\frac{\dd^nF}{\dd\phi^n}\bigg\|_{L^2(\bfR)}<\infty,\quad n=0,1,2,
\ee
that is, $F$ belongs to $H^2(\bfR)$. By the Fourier transform $\hat{F}(s)$ defined by
\berr
\hat{F}(s)=\frac{1}{2\pi}\int_\bfR F(\phi)\e^{-\ii\phi s}\dd\phi
\eerr
where
\be\label{3.24}
F(\phi)=\int_\bfR\hat{F}(s)\e^{\ii\phi s}\dd s,
\ee
it follows from Plancherel theorem that $\|s^2\hat{F}(s)\|_{L^2(\bfR)}<\infty$. Equations \eqref{3.22} and \eqref{3.24} imply
\be\label{3.25}
f(k)=\frac{k-k_2}{(k-\ii)^3}\int_\bfR\hat{F}(s)\e^{\ii\phi s}\dd s,\quad k>k_2.
\ee
Writing
$$f(k)=f_a(x,t,k)+f_r(x,t,k),\quad t>0,~k>k_2,$$
where the functions $f_a$ and $f_r$ are defined by
\bea
f_a(x,t,k)&=&\frac{k-k_2}{(k-\ii)^3}\int_{-\frac{t}{4}}^\infty\hat{F}(s)\e^{s\Phi(k)}\dd s,~\quad t>0,~k\in\Om^1_1,\nn\\
f_r(x,t,k)&=&\frac{k-k_2}{(k-\ii)^3}\int^{-\frac{t}{4}}_{-\infty}\hat{F}(s)\e^{s\Phi(k)}\dd s,\quad t>0,~k>k_2,\nn
\eea
we infer that $f_a(x,t,\cdot)$ is continuous in $\bar{\Om}^1_1$ and analytic in $\Om^1_1$. Moreover, we can get
\bea\label{3.26}
|f_a(x,t,k)|&\leq&\frac{|k-k_2|}{|k-\ii|^3}\|\hat{F}(s)\|_{L^1(\bfR)}\sup_{s\geq-\frac{t}{4}}\e^{s\text{Re}\Phi(k)}\nn\\
&\leq&\frac{C|k-k_2|}{|k-\ii|^3}\e^{\frac{t}{4}|\text{Re}\Phi(k)|},\quad t>0,~k\in\bar{\Om}^1_1,~\xi\in\mathcal{I}.
\eea
Furthermore, we have
\bea\label{3.27}
|f_r(x,t,k)|&\leq&\frac{|k-k_2|}{|k-\ii|^3}\int_{-\infty}^{-\frac{t}{4}}s^2|\hat{F}(s)|s^{-2}\dd s\nn\\
&\leq&\frac{C}{1+|k|^2}\|s^2\hat{F}(s)\|_{L^2(\bfR)}\sqrt{\int_{-\infty}^{-\frac{t}{4}}s^{-4}\dd s},\\
&\leq&\frac{C}{1+|k|^2}t^{-3/2},\quad t>0,~k>k_2,~\xi\in\mathcal{I}.\nn
\eea
Hence, the $L^1,L^2$ and $L^\infty$ norms of $f_r$ on $(k_2,\infty)$ are $O(t^{-3/2})$. Letting
\bea
\begin{aligned}
r_{1,a}(x,t,k)&=f_0(k)+f_a(x,t,k),~\quad t>0,~ k\in\bar{\Om}_1^1,\label{3.28}\\
r_{1,r}(x,t,k)&=f_r(x,t,k),~~\qquad\qquad t>0,~ k>k_2.
\end{aligned}
\eea
For $k<-k_2$, we use the symmetry \eqref{3.16} to extend this decomposition.

We next derive the decomposition of $r_1(k)$ for $-k_1<k<k_1$. Following \cite{PD,JL3}, we split $r_1(k)$ into even and odd parts as follows:
\be\label{3.29}
r_1(k)=r_+(k^2)+kr_-(k^2),\quad k\in\bfR,
\ee
where $r_\pm:[0,\infty)\rightarrow\bfC$ are defined by
\berr
r_+(s)=\frac{r_1(\sqrt{s})+r_1(-\sqrt{s})}{2},~r_-(s)=\frac{r_1(\sqrt{s})-r_1(-\sqrt{s})}{2\sqrt{s}},~s\geq0.
\eerr
We write $r_1(k)$ as the following form of Taylor series
\be
r_1(k)=\sum_{j=0}^{10}q_jk^j+\frac{1}{10!}\int_0^kr_1^{(11)}(t)(k-t)^{10}\dd t,~q_j=\frac{r_1^{(j)}(0)}{j!}.
\ee
It then follows that
\bea\label{3.31}
\begin{aligned}
r_+(s)&=\sum_{j=0}^{5}q_{2j}s^j+\frac{1}{2\times10!}\int_0^{\sqrt{s}}(r_1^{(11)}(t)-r_1^{(11)}(-t))(\sqrt{s}-t)^{10}\dd t,\\
r_-(s)&=\sum_{j=0}^{4}q_{2j+1}s^j+\frac{1}{2\times10!\sqrt{s}}\int_0^{\sqrt{s}}(r_1^{(11)}(t)+r_1^{(11)}(-t))(\sqrt{s}-t)^{10}\dd t.
\end{aligned}
\eea
 Letting $\{p_j^\pm\}_0^4$ denote the coefficients of the Taylor series representations
\berr
r_\pm(k^2)=\sum_{j=0}^4p_j^\pm(k^2-k_1^2)^j+\frac{1}{4!}\int_{k_1^2}^{k^2}r_\pm^{(5)}(t)(k^2-t)^4\dd t,
\eerr
we infer that the function $f_0(k)$ defined by
\be
f_0(k)=\sum_{j=0}^4p_j^+(k^2-k_1^2)^j+k\sum_{j=0}^4p_j^-(k^2-k_1^2)^j
\ee
has the following properties:\\
(i) $f_0(k)$ is a polynomial in $k\in\bfC$ whose coefficients are bounded.\\
(ii) The difference $f(k)=r_1(k)-f_0(k)$, which satisfies
\be\label{3.33}
\frac{\dd^n}{\dd k^n}f(k)\leq C|k^2-k_1^2|^{5-n},~-k_1<k< k_1,~\xi\in\mathcal{I},~n=0,1,2,
\ee
where $C$ is independent of $\xi,k$. The decomposition of $r_1(k)$ for $-k_1<k< k_1$ can now be derived as follows.
Since the function $k\mapsto\phi$ defined in \eqref{3.21} is a bijection $(-k_1,k_1)\rightarrow(128\beta k_1^5-16\alpha k_1^3,-128\beta k_1^5+16\alpha k_1^3)$ (see Fig. \ref{fig4}), we may define a function $F(\phi)$ by
\be\label{3.34}
F(\phi)=\left\{
\begin{aligned}
&\frac{1}{k^2-k_1^2}f(k),~|\phi|<-128\beta k_1^5+16\alpha k_1^3,\\
&0,\qquad\qquad \quad |\phi|\geq-128\beta k_1^5+16\alpha k_1^3.
\end{aligned}
\right.
\ee
Thus, we have
\be\label{3.35}
F^{(n)}(\phi)=\bigg(\frac{1}{160\beta(k^2-k_1^2)(k^2+k_1^2-\frac{3\alpha}{20\beta})}\frac{\partial}{\partial k}\bigg)^n\frac{f(k)}{k^2-k_1^2},~|\phi|<-128\beta k_1^5+16\alpha k_1^3.
\ee
Equations \eqref{3.33} and \eqref{3.35} imply that
\berr
\bigg|\frac{\dd^n}{\dd\phi^n}F(\phi)\bigg|\leq C,~|\phi|<-128\beta k_1^5+16\alpha k_1^3,~n=0,1,2.
\eerr
Therefore, $F(\phi)$ satisfies \eqref{3.23}. On the other hand, \eqref{3.24} and \eqref{3.34} imply
\be
(k^2-k_1^2)\int_\bfR\hat{F}(s)\e^{s\Phi(k)}\dd s=\left\{
\begin{aligned}
&f(k),\quad |k|<k_1,\\
&0,\qquad ~|k|\geq k_1.
\end{aligned}
\right.
\ee
Letting
$$f(k)=f_a(x,t,k)+f_r(x,t,k),\quad t>0,~|k|k_1,~\xi\in\mathcal{I},$$
where the functions $f_a$ and $f_r$ are defined by
\bea
f_a(x,t,k)&=&(k^2-k^2_1)\int_{-\frac{t}{4}}^\infty\hat{F}(s)\e^{s\Phi(k)}\dd s,~\quad t>0,~k\in\Om^3_1,\nn\\
f_r(x,t,k)&=&(k^2-k^2_1)\int^{-\frac{t}{4}}_{-\infty}\hat{F}(s)\e^{s\Phi(k)}\dd s,\quad t>0,~|k|<k_1,\nn
\eea
we infer that $f_a(x,t,\cdot)$ is continuous in $\bar{\Om}^3_1$ and analytic in $\Om^3_1$. Moreover, we can get
from \eqref{3.26} and \eqref{3.27} that
\bea
|f_a(x,t,k)|&\leq&C|k^2-k^2_1|\e^{\frac{t}{4}|\text{Re}\Phi(k)|},\quad t>0,~k\in\bar{\Om}^3_1,~\xi\in\mathcal{I},\nn\\
|f_r(x,t,k)|&\leq&Ct^{-3/2},\qquad\qquad\qquad~ t>0,~|k|<k_1,~\xi\in\mathcal{I}.\nn
\eea
It follows  from
\berr
r_{1,a}(x,t,k)=f_0(k)+f_a(x,t,k),\quad r_{1,r}(x,t,k)=f_r(x,t,k)
\eerr
that one get a decomposition of $r_1(k)$ for $|k|<k_1$ with the properties listed in the statement of the lemma.
Thus, we find a decomposition of $r_1(k)$ for $|k|>k_2$ and $|k|<k_1$ with the properties listed in the statement of the lemma. The decomposition of the function $r_3(k)$ can be obtained by a similar procedure as the decomposition of $r_1(k)$ for $|k|<k_1$.

Finally, the decompositions of $r_2(k)$ and $r_4(k)$ can be obtained from $r_3(k)$ and $r_1(k)$ by Schwartz conjugation.
\end{proof}
\begin{figure}[htbp]
  \centering
  \includegraphics[width=3.5in]{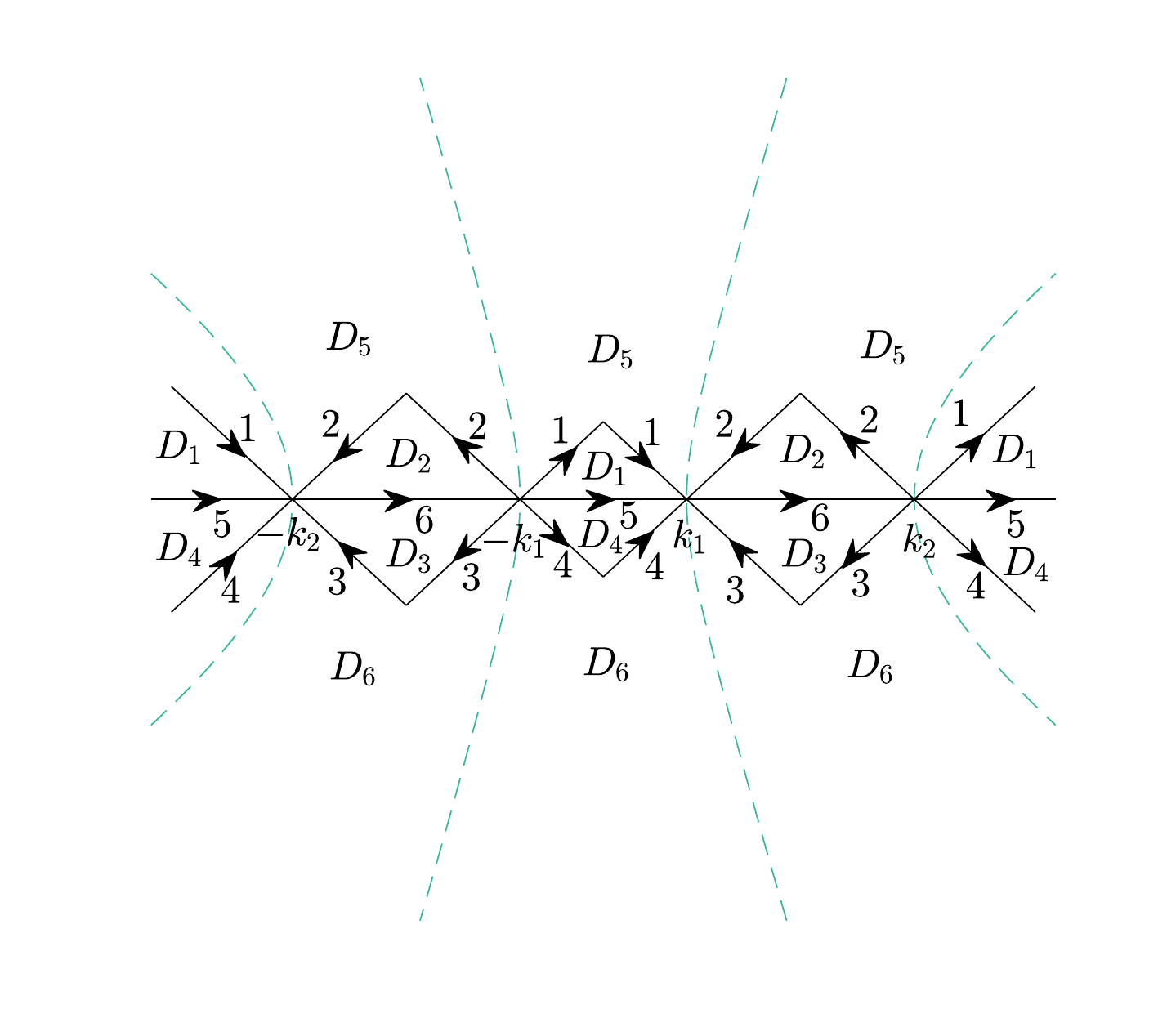}
  \caption{The oriented contour $\Sigma$ and the open sets $\{D_j\}_1^6$ in the complex $k$-plane.}\label{fig5}
\end{figure}

The purpose of the next transformation is to deform the contour so that the jump matrix involves the exponential factor $\e^{-t\Phi}$ on the parts of the contour where Re$\Phi$ is positive and the factor $\e^{t\Phi}$ on the parts where Re$\Phi$ is negative. More precisely, we put
\be\label{3.37}
M^{(2)}(x,t;k)=M^{(1)}(x,t;k)G(k),
\ee
where
\be\label{3.38}
G(k)=\left\{
\begin{aligned}
&\begin{pmatrix}
1 &~ 0\\[4pt]
-r_{1,a}\delta^{-2}\e^{t\Phi} ~& 1
\end{pmatrix},~~\quad k\in D_1,\\
&\begin{pmatrix}
1 ~& -r_{2,a}\delta^{2}\e^{-t\Phi}\\[4pt]
0 ~& 1
\end{pmatrix},\qquad~ k\in D_2,\\
&\begin{pmatrix}
1 ~& 0\\[4pt]
r_{3,a}\delta^{-2}\e^{t\Phi} ~& 1
\end{pmatrix},\qquad~~~ k\in D_3,\\
&\begin{pmatrix}
1 ~& r_{4,a}\delta^2\e^{-t\Phi}\\[4pt]
0 ~& 1
\end{pmatrix},~~~\qquad k\in D_4,\\
&I,~~~~~~~~\qquad\qquad\qquad\quad k\in D_5\cup D_6.
\end{aligned}
\right.
\ee
Then the matrix $M^{(2)}(x,t;k)$ satisfies the following RH problem
\be\label{3.39}
M_+^{(2)}(x,t;k)=M_-^{(2)}(x,t;k)J^{(2)}(x,t;k),\quad k\in\Sigma,
\ee
with the jump matrix $J^{(2)}=G_-^{-1}(k)J^{(1)}G_+(k)$ is given by
\bea
J^{(2)}_1&=&\begin{pmatrix}
1 & 0\\[4pt]
r_{1,a}\delta^{-2}\e^{t\Phi} & 1
\end{pmatrix},~\qquad J^{(2)}_2=\begin{pmatrix}
1 &~ -r_{2,a}\delta^2\e^{-t\Phi}\\
0~& 1
\end{pmatrix},\nn\\
J^{(2)}_3&=&\begin{pmatrix}
1 &~ 0\\
-r_{3,a}\delta^{-2}\e^{t\Phi} ~& 1
\end{pmatrix},\quad
J^{(2)}_4=\begin{pmatrix}
1 &~ r_{4,a}\delta^2\e^{-t\Phi}\\
0~& 1
\end{pmatrix},\label{3.40}\\
J^{(2)}_5&=&\begin{pmatrix}
1 ~& r_{4,r}\delta^{2}\e^{-t\Phi} \\[4pt]
0 ~& 1 \\
\end{pmatrix}\begin{pmatrix}
1 ~& 0 \\[4pt]
r_{1,r}\delta^{-2}\e^{t\Phi} ~& 1 \\
\end{pmatrix}, \quad J^{(2)}_{6}=\begin{pmatrix}
1 ~& 0 \\[4pt]
r_{3,r}\delta_-^{-2}\e^{t\Phi} ~& 1
\end{pmatrix}\begin{pmatrix}
 1 ~& r_{2,r}\delta_+^{2}\e^{-t\Phi} \\[4pt]
0 ~& 1 \\
\end{pmatrix},\nn
\eea
with $J^{(2)}_i$ denoting the restriction of $J^{(2)}$ to the contour $\Sigma$ labeled by $i$ in Fig. \ref{fig5}. It is easy to see that the jump matrix $J^{(2)}$ decays to identity matrix $I$ as $t\rightarrow\infty$ everywhere except near the critical points $\pm k_1$ and $\pm k_2$. This implies that we only need to consider a neighborhood of the critical points $\pm k_1$ and $\pm k_2$ when we study the long-time asymptotics of $M^{(2)}(x,t;k)$ in terms of the corresponding RH problem.

\subsection{Local models near the critical points $\pm k_1$ and $\pm k_2$}
 We introduce the following scaling operators
\bea
\begin{aligned}
&S_{-k_2}:~k\mapsto\frac{z}{4\sqrt{t(40\beta k_2^3-3\alpha k_2)}}-k_2,\\
&S_{-k_1}:~k\mapsto\frac{z}{4\sqrt{t(3\alpha k_1-40\beta k_1^3)}}-k_1,\\
&S_{k_1}:~k\mapsto\frac{z}{4\sqrt{t(3\alpha k_1-40\beta k_1^3)}}+k_1,\\
&S_{k_2}:~k\mapsto\frac{z}{4\sqrt{t(40\beta k_2^3-3\alpha k_2)}}+k_2.
\end{aligned}
\eea
Integrating by parts in formula \eqref{3.9} yields,
\be\label{3.42}
\begin{aligned}
\delta(k)&=\bigg(\frac{k-k_2}{k-k_1}\bigg)^{-\ii\nu(k_1)}\bigg(\frac{k+k_1}{k+k_2}\bigg)^{-\ii\nu(k_1)}\e^{\chi_1(k)}\\
&=\bigg(\frac{k-k_2}{k-k_1}\bigg)^{-\ii\nu(k_2)}\bigg(\frac{k+k_1}{k+k_2}\bigg)^{-\ii\nu(k_2)}\e^{\chi_2(k)},
\end{aligned}
\ee
where
\bea
\nu(k_1)&=&\frac{1}{2\pi}\ln(1+|r(k_1)|^2)>0,\label{3.43}\\
\chi_1(k)&=&\frac{1}{2\pi\ii}\int_{k_1}^{k_2}\ln\bigg(\frac{1+|r(s)|^2}{1+|r(k_1)|^2}\bigg)
\bigg(\frac{1}{s-k}-\frac{1}{s+k}\bigg)\dd s,\label{3.44}\\
\nu(k_2)&=&\frac{1}{2\pi}\ln(1+|r(k_2)|^2)>0,\label{3.45}\\
\chi_2(k)&=&\frac{1}{2\pi\ii}\int_{k_1}^{k_2}\ln\bigg(\frac{1+|r(s)|^2}{1+|r(k_2)|^2}\bigg)
\bigg(\frac{1}{s-k}-\frac{1}{s+k}\bigg)\dd s.\label{3.46}
\eea
Hence, we have
\bea
S_{-k_2}(\delta(k)\e^{-\frac{t\Phi(k)}{2}})&=&\delta_{-k_2}^0(z)\delta_{-k_2}^1(z),\nn\\
S_{-k_1}(\delta(k)\e^{-\frac{t\Phi(k)}{2}})&=&\delta_{-k_1}^0(z)\delta_{-k_1}^1(z),\nn\\
S_{k_1}(\delta(k)\e^{-\frac{t\Phi(k)}{2}})&=&\delta_{k_1}^0(z)\delta_{k_1}^1(z),\nn\\
S_{k_2}(\delta(k)\e^{-\frac{t\Phi(k)}{2}})&=&\delta_{k_2}^0(z)\delta_{k_2}^1(z),\nn
\eea
with
\bea
\delta_{-k_2}^0(z)&=&\bigg(16t(k_2-k_1)^2(40\beta k_2^3-3\alpha k_2)\bigg)^{-\frac{\ii\nu(k_2)}{2}}\bigg(\frac{2k_2}{k_1+k_2}\bigg)^{-\ii\nu(k_2)}\nn\\
&&\times\e^{\chi_2(-k_2)}\e^{-8\ii k_2^3t(8\beta k_2^2-\alpha)},\label{3.47}\\
\delta_{-k_2}^1(z)&=&(-z)^{\ii\nu(k_2)}\e^{(\chi_2(z/\sqrt{16tk_2(40\beta k_2^2-3\alpha)}-k_2)-\chi_2(-k_2))}\nn\\
&&\times\bigg(\frac{-z/\sqrt{16tk_2(40\beta k_2^2-3\alpha)}-k_1+k_2}{k_2-k_1}\bigg)^{-\ii\nu(k_2)}\nn\\
&&\times\bigg(\frac{k_1+k_2}{-z/\sqrt{16tk_2(40\beta k_2^2-3\alpha)}+k_1+k_2}\bigg)^{-\ii\nu(k_2)}\label{3.48}\\
&&\times\bigg(\frac{-z/\sqrt{16tk_2(40\beta k_2^2-3\alpha)}+2k_2}{2k_2}\bigg)^{-\ii\nu(k_2)}\nn\\
&&\times\exp\bigg(\frac{\ii z^2}{4}\bigg[1-\frac{(40\beta k_2^2-\alpha)z}{4\sqrt{t}(40\beta k_2^3-3\alpha k_2)^{3/2}}+\frac{5\beta z^2}{4tk_2(40\beta k_2^2-3\alpha)^2}\nn\\
&&\qquad\qquad-\frac{\beta z^3}{16t^{3/2}(40\beta k_2^3-3\alpha k_2)^{5/2}}\bigg]\bigg),\nn
\eea
\bea
\delta_{-k_1}^0(z)&=&\bigg(16t(k_2-k_1)^2(3\alpha k_1-40\beta k_1^3)\bigg)^{\frac{\ii\nu(k_1)}{2}}\bigg(\frac{k_1+k_2}{2k_1}\bigg)^{-\ii\nu(k_1)}\nn\\
&&\times\e^{\chi_1(-k_1)}\e^{-8\ii k_1^3t(8\beta k_1^2-\alpha)},\label{3.49}\\
\delta_{-k_1}^1(z)&=&z^{-\ii\nu(k_1)}\e^{(\chi_1(z/\sqrt{16tk_1(3\alpha-40\beta k_1^2)}-k_1)-\chi_1(-k_1))}\nn\\
&&\times\bigg(\frac{k_2-k_1}{z/\sqrt{16tk_1(3\alpha-40\beta k_1^2)}-k_1+k_2}\bigg)^{-\ii\nu(k_1)}\nn\\
&&\times\bigg(\frac{-z/\sqrt{16tk_1(3\alpha-40\beta k_1^2)}+k_1+k_2}{k_1+k_2}\bigg)^{-\ii\nu(k_1)}\label{3.50}\\
&&\times\bigg(\frac{2k_1}{-z/\sqrt{16tk_1(3\alpha-40\beta k_1^2)}+2k_1}\bigg)^{-\ii\nu(k_1)}\nn\\
&&\times\exp\bigg(-\frac{\ii z^2}{4}\bigg[1+\frac{(40\beta k_1^2-\alpha)z}{4\sqrt{t}(3\alpha k_1-40\beta k_1^3)^{3/2}}-\frac{5\beta z^2}{4tk_1(3\alpha-40\beta k_1^2)^2}\nn\\
&&\qquad\qquad+\frac{\beta z^3}{16t^{3/2}(3\alpha k_1-40\beta k_1^3)^{5/2}}\bigg]\bigg),\nn
\eea
\bea
\delta_{k_1}^0(z)&=&\bigg(16t(k_2-k_1)^2(3\alpha k_1-40\beta k_1^3)\bigg)^{-\frac{\ii\nu(k_1)}{2}}\bigg(\frac{2k_1}{k_1+k_2}\bigg)^{-\ii\nu(k_1)}\nn\\
&&\times\e^{\chi_1(k_1)}\e^{8\ii k_1^3t(8\beta k_1^2-\alpha)},\label{3.51}\\
\delta_{k_1}^1(z)&=&(-z)^{\ii\nu(k_1)}\e^{(\chi_1(z/\sqrt{16tk_1(3\alpha-40\beta k_1^2)}+k_1)-\chi_1(k_1))}\nn\\
&&\times\bigg(\frac{-z/\sqrt{16tk_1(3\alpha-40\beta k_1^2)}-k_1+k_2}{k_2-k_1}\bigg)^{-\ii\nu(k_1)}\nn\\
&&\times\bigg(\frac{k_1+k_2}{z/\sqrt{16tk_1(3\alpha-40\beta k_1^2)}+k_1+k_2}\bigg)^{-\ii\nu(k_1)}\label{3.52}\\
&&\times\bigg(\frac{z/\sqrt{16tk_1(3\alpha-40\beta k_1^2)}+2k_1}{2k_1}\bigg)^{-\ii\nu(k_1)}\nn\\
&&\times\exp\bigg(\frac{\ii z^2}{4}\bigg[1-\frac{(40\beta k_1^2-\alpha)z}{4\sqrt{t}(3\alpha k_1-40\beta k_1^3)^{3/2}}-\frac{5\beta z^2}{4tk_1(3\alpha-40\beta k_1^2)^2}\nn\\
&&\qquad\qquad-\frac{\beta z^3}{16t^{3/2}(3\alpha k_1-40\beta k_1^3)^{5/2}}\bigg]\bigg),\nn\\
\delta_{k_2}^0(z)&=&\bigg(16t(k_2-k_1)^2(40\beta k_2^3-3\alpha k_2)\bigg)^{\frac{\ii\nu(k_2)}{2}}\bigg(\frac{k_1+k_2}{2k_2}\bigg)^{-\ii\nu(k_2)}\nn\\
&&\times\e^{\chi_2(k_2)}\e^{8\ii k_2^3t(8\beta k_2^2-\alpha)},\label{3.53}\\
\delta_{k_2}^1(z)&=&z^{-\ii\nu(k_2)}\e^{(\chi_2(z/\sqrt{16tk_2(40\beta k_2^2-3\alpha)}+k_2)-\chi_2(k_2))}\nn\\
&&\times\bigg(\frac{k_2-k_1}{z/\sqrt{16tk_2(40\beta k_2^2-3\alpha)}-k_1+k_2}\bigg)^{-\ii\nu(k_2)}\nn\\
&&\times\bigg(\frac{2k_2}{z/\sqrt{16tk_2(40\beta k_2^2-3\alpha)}+2k_2}\bigg)^{-\ii\nu(k_2)}\label{3.54}\\
&&\times\bigg(\frac{z/\sqrt{16tk_2(40\beta k_2^2-3\alpha)}+k_1+k_2}{k_1+k_2}\bigg)^{-\ii\nu(k_2)}\nn\\
&&\times\exp\bigg(-\frac{\ii z^2}{4}\bigg[1+\frac{(40\beta k_2^2-\alpha)z}{4\sqrt{t}(40\beta k_2^3-3\alpha k_2)^{3/2}}+\frac{5\beta z^2}{4tk_2(40\beta k_2^2-3\alpha)^2}\nn\\
&&\qquad\qquad+\frac{\beta z^3}{16t^{3/2}(40\beta k_2^3-3\alpha k_2)^{5/2}}\bigg]\bigg).\nn
\eea
For $j=1,2$, let $D_\varepsilon(\pm k_j)$ denote the open disk of radius $\varepsilon$ centered at $\pm k_j$ for a small $\varepsilon>0$.  Now we define
\bea
\tilde{\tilde{M}}(x,t;z)&=&M^{(2)}(x,t;k)(\delta_{-k_2}^0)^{\sigma_3}(z),\quad k\in D_\varepsilon(-k_2)\setminus\Sigma,\nn\\
\check{\check{M}}(x,t;z)&=&M^{(2)}(x,t;k)(\delta_{-k_1}^0)^{\sigma_3}(z),\quad k\in D_\varepsilon(-k_1)\setminus\Sigma,\nn\\
\check{M}(x,t;z)&=&M^{(2)}(x,t;k)(\delta_{k_1}^0)^{\sigma_3}(z),\quad~~ k\in D_\varepsilon(k_1)\setminus\Sigma,\nn\\
\tilde{M}(x,t;z)&=&M^{(2)}(x,t;k)(\delta_{k_2}^0)^{\sigma_3}(z),\quad~~ k\in D_\varepsilon(k_2)\setminus\Sigma.\nn
\eea
Then $\check{\check{M}}$, $\check{M}$, $\tilde{\tilde{M}}$ and $\tilde{M}$ are the sectionally analytic functions of $z$ which satisfy
\bea
\tilde{\tilde{M}}_+(x,t;z)&=&\tilde{\tilde{M}}_-(x,t;z)\tilde{\tilde{J}}(x,t;z),\quad k\in \mathcal{X}_{-k_2}^\varepsilon,\nn\\
\check{\check{M}}_+(x,t;z)&=&\check{\check{M}}_-(x,t;z)\check{\check{J}}(x,t;z),\quad k\in \mathcal{X}_{-k_1}^\varepsilon,\nn\\
\check{M}_+(x,t;z)&=&\check{M}_-(x,t;z)\check{J}(x,t;z),\quad k\in \mathcal{X}_{k_1}^\varepsilon,\nn\\
\tilde{M}_+(x,t;z)&=&\tilde{M}_-(x,t;z)\tilde{J}(x,t;z),\quad k\in \mathcal{X}_{k_2}^\varepsilon,\nn
\eea
where $\mathcal{X}_{\pm k_j}=X\pm k_j$ denote the cross $X$ defined by \eqref{2.19} centered at $\pm k_j$ and $\mathcal{X}_{\pm k_j}^\varepsilon=\mathcal{X}_{\pm k_j}\cap D_\varepsilon(\pm k_j)$ for $j=1,2$. Moreover, the corresponding jump matrices are given by
\be\label{3.55}
\tilde{\tilde{J}}(x,t;z)=\left\{
\begin{aligned}
&\begin{pmatrix}
1 ~& r_{2,a}(\delta_{-k_2}^1)^{2}\\[4pt]
0 ~& 1
\end{pmatrix},~\qquad k\in (\mathcal{X}_{-k_2}^\varepsilon)_1,\\
&\begin{pmatrix}
1 ~& 0\\[4pt]
-r_{1,a}(\delta_{-k_2}^1)^{-2} ~& 1
\end{pmatrix},\quad k\in (\mathcal{X}_{-k_2}^\varepsilon)_2,\\
&\begin{pmatrix}
1 ~& -r_{4,a}(\delta_{-k_2}^1)^{2}\\[4pt]
0 ~& 1
\end{pmatrix},~~\quad k\in (\mathcal{X}_{-k_2}^\varepsilon)_3,\\
&\begin{pmatrix}
1 ~& 0\\[4pt]
r_{3,a}(\delta_{-k_2}^1)^{-2} ~& 1
\end{pmatrix},\qquad k\in (\mathcal{X}_{-k_2}^\varepsilon)_4,\\
\end{aligned}
\right.
\ee
\be\label{3.56}
\check{\check{J}}(x,t;z)=\left\{
\begin{aligned}
&\begin{pmatrix}
1 ~& 0\\[4pt]
r_{1,a}(\delta_{-k_1}^1)^{-2} ~& 1
\end{pmatrix},~\quad k\in (\mathcal{X}_{-k_1}^\varepsilon)_1,\\
&\begin{pmatrix}
1 ~& -r_{2,a}(\delta_{-k_1}^1)^{2}\\[4pt]
0 ~& 1
\end{pmatrix},~~~~ k\in (\mathcal{X}_{-k_1}^\varepsilon)_2,\\
&\begin{pmatrix}
1 ~& 0\\[4pt]
-r_{3,a}(\delta_{-k_1}^1)^{-2} ~& 1
\end{pmatrix},~~ k\in (\mathcal{X}_{-k_1}^\varepsilon)_3,\\
&\begin{pmatrix}
1 ~& r_{4,a}(\delta_{-k_1}^1)^{2}\\[4pt]
0 ~& 1
\end{pmatrix},\quad\quad k\in (\mathcal{X}_{-k_1}^\varepsilon)_4,\\
\end{aligned}
\right.
\ee
\be\label{3.57}
\check{J}(x,t;z)=\left\{
\begin{aligned}
&\begin{pmatrix}
1 ~& r_{2,a}(\delta_{k_1}^1)^{2}\\[4pt]
0 ~& 1
\end{pmatrix},~\qquad k\in (\mathcal{X}_{k_1}^\varepsilon)_1,\\
&\begin{pmatrix}
1 ~& 0\\[4pt]
-r_{1,a}(\delta_{k_1}^1)^{-2} ~& 1
\end{pmatrix},\quad k\in (\mathcal{X}_{k_1}^\varepsilon)_2,\\
&\begin{pmatrix}
1 ~& -r_{4,a}(\delta_{k_1}^1)^{2}\\[4pt]
0 ~& 1
\end{pmatrix},~~\quad k\in (\mathcal{X}_{k_1}^\varepsilon)_3,\\
&\begin{pmatrix}
1 ~& 0\\[4pt]
r_{3,a}(\delta_{k_1}^1)^{-2} ~& 1
\end{pmatrix},\qquad k\in (\mathcal{X}_{k_1}^\varepsilon)_4,\\
\end{aligned}
\right.
\ee
and
\be\label{3.58}
\tilde{J}(x,t;z)=\left\{
\begin{aligned}
&\begin{pmatrix}
1 ~& 0\\[4pt]
r_{1,a}(\delta_{k_2}^1)^{-2} ~& 1
\end{pmatrix},~\quad k\in (\mathcal{X}_{k_2}^\varepsilon)_1,\\
&\begin{pmatrix}
1 ~& -r_{2,a}(\delta_{k_2}^1)^{2}\\[4pt]
0 ~& 1
\end{pmatrix},~~~~ k\in (\mathcal{X}_{k_2}^\varepsilon)_2,\\
&\begin{pmatrix}
1 ~& 0\\[4pt]
-r_{3,a}(\delta_{k_2}^1)^{-2} ~& 1
\end{pmatrix},~~ k\in (\mathcal{X}_{k_2}^\varepsilon)_3,\\
&\begin{pmatrix}
1 ~& r_{4,a}(\delta_{k_2}^1)^{2}\\[4pt]
0 ~& 1
\end{pmatrix},\quad\quad k\in (\mathcal{X}_{k_2}^\varepsilon)_4.\\
\end{aligned}
\right.
\ee
For the jump matrix $\tilde{J}(x,t;z)$, define $$q=r(k_2),$$ then for any fixed $z\in X$, we have $k(z)\rightarrow k_2$ as $t\rightarrow\infty$. Hence,
$$r_{1,a}(k)\rightarrow q,\quad r_{2,a}(k)\rightarrow\frac{\bar{q}}{1+|q|^2},\quad \delta_{k_2}^1\rightarrow\e^{-\frac{\ii z^2}{4}}z^{-\ii\nu(q)}.$$
This implies that the jump matrix $\tilde{J}$ tend to the matrix $J^X$ defined in \eqref{2.21} for large $t$. In other words, the jumps of $M^{(2)}$ for $k$ near $k_2$ approach those of the function $M^X(\delta_{k_2}^0)^{-\sigma_3}$ as $t\rightarrow\infty$. Therefore, we can approximate $M^{(2)}$ in the neighborhood $D_\varepsilon(k_2)$ of $k_2$ by
\be\label{3.59}
M^{(k_2)}(x,t;k)=(\delta_{k_2}^0)^{\sigma_3}M^X(q,z)(\delta_{k_2}^0)^{-\sigma_3},
\ee
where $M^X(q,z)$ is given by \eqref{2.22}. On the other hand, according to the symmetry property \eqref{3.3}, one can deduce that \berr
\tilde{\tilde{J}}(x,t;z)\rightarrow \overline{J^X(q,-\bar{z})},\quad \text{as}~t\rightarrow\infty.
\eerr
Thus, by uniqueness of the RH problem, we can approximate $M^{(2)}$ in the neighborhood $D_\varepsilon(-k_2)$ of $-k_2$ by
\be\label{3.60}
M^{(-k_2)}(x,t;k)=(\delta_{-k_2}^0)^{\sigma_3}\overline{M^X(q,-\bar{z})}(\delta_{-k_2}^0)^{-\sigma_3}.
\ee

For the case of $\check{J}(x,t;z)$, as $t\rightarrow\infty$, we find
\berr
r_{1,a}(k)\rightarrow r(k_1),\quad r_{2,a}(k)\rightarrow\frac{\overline{r(k_1)}}{1+|r(k_1)|^2},\quad \delta_{k_1}^1\rightarrow\e^{\frac{\ii z^2}{4}}(-z)^{\ii\nu(k_1)}.
\eerr
This implies as $t\rightarrow\infty$ that
\berr
\check{J}(x,t;z)\rightarrow J^{Y}(p,z)=\left\{
\begin{aligned}
&\begin{pmatrix}
1 ~& \frac{\bar{p}}{1+|p|^2}\e^{\frac{\ii z^2}{2}}(-z)^{2\ii\nu(p)}\\[4pt]
0 ~& 1
\end{pmatrix},~~~\quad z\in X_1,\\
&\begin{pmatrix}
1 ~& 0\\[4pt]
-p\e^{-\frac{\ii z^2}{2}}(-z)^{-2\ii\nu(p)} ~& 1
\end{pmatrix},~\quad\quad z\in X_2,\\
&\begin{pmatrix}
1 ~& -\bar{p}\e^{\frac{\ii z^2}{2}}(-z)^{2\ii\nu(p)}\\[4pt]
0 ~& 1
\end{pmatrix},~~~~\qquad z\in X_3,\\
&\begin{pmatrix}
1 ~& 0\\[4pt]
\frac{p}{1+|p|^2}\e^{-\frac{\ii z^2}{2}}(-z)^{-2\ii\nu(p)} ~& 1
\end{pmatrix},~~~z\in X_4,\\
\end{aligned}
\right.
\eerr
if we set $$p=r(k_1).$$ It is easy to verify that
\berr
J^{Y}(p,z)=\overline{J^X(\bar{p},-\bar{z})},
\eerr
which in turn implies that
\be\label{3.61}
M^{Y}(p,z)=\overline{M^X(\bar{p},-\bar{z})},
\ee
where $M^{Y}(p,z)$ is the unique solution of the following RH problem
\berr\left\{
\begin{aligned}
&M^{Y}_+(p,z)=M^{Y}_-(p,z)J^{Y}(p,z),~\text{for~almost~every}~z\in X,\\
&M^{Y}(p,z)\rightarrow I,~~~\qquad\qquad\qquad\text{as}~z\rightarrow\infty.
\end{aligned}
\right.
\eerr
Therefore, we find that
\berr
M^{Y}(p,z)=I-\frac{\ii}{z}\begin{pmatrix}
0 ~& \beta^{Y}(p)\\[4pt]
\overline{\beta^{Y}(p)} &~ 0
\end{pmatrix}+O\bigg(\frac{p}{z^2}\bigg),
\eerr
where
\berr
\beta^{Y}(p)=\sqrt{\nu(p)}\e^{-\ii(\frac{\pi}{4}+\arg p+\arg\Gamma(-\ii\nu(p)))}.
\eerr
As a consequence, we can approximate $M^{(2)}(x,t;k)$ in the neighborhood $D_\varepsilon(k_1)$ of $k_1$ by
\be\label{3.62}
M^{(k_1)}(x,t;k)=(\delta_{k_1}^0)^{\sigma_3}M^{Y}(p,z)(\delta_{k_1}^0)^{-\sigma_3}.
\ee
Using again \eqref{3.3}, we can use
\be\label{3.63}
M^{(-k_1)}(x,t;k)=(\delta_{-k_1}^0)^{\sigma_3}\overline{M^{Y}(p,-\bar{z})}(\delta_{-k_1}^0)^{-\sigma_3}
\ee
to approximate $M^{(2)}(x,t;k)$ in the neighborhood $D_\varepsilon(-k_1)$ of $-k_1$.

\begin{lemma}\label{lem3}
For each $t>0$, $\xi\in\mathcal{I}$ and $j=1,2$, the functions $M^{(\pm k_j)}(x,t;k)$ defined in \eqref{3.62}, \eqref{3.63}, \eqref{3.59} and \eqref{3.60} are analytic functions of $k\in D_\varepsilon(\pm k_j)\setminus\mathcal{X}_{\pm k_j}^\varepsilon$. Furthermore,
\be\label{3.64}
|M^{(\pm k_j)}(x,t;k)-I|\leq C,\quad t>3,~\xi\in\mathcal{I},~k\in\overline{ D_\varepsilon(\pm k_j)}\setminus\mathcal{X}_{\pm k_j}^\varepsilon,~j=1,2.
\ee
Across $\mathcal{X}_{\pm k_j}^\varepsilon$, $M^{(\pm k_j)}(x,t;k)$ satisfied the jump condition $M_+^{(\pm k_j)}=M_-^{(\pm k_j)}J^{(\pm k_j)}$, where the jump matrix $J^{(\pm k_j)}$ satisfy the following estimates for $1\leq n\leq\infty$:
\be\label{3.65}
\|J^{(2)}-J^{(\pm k_j)}\|_{L^n(\mathcal{X}_{\pm k_j}^\varepsilon)}\leq Ct^{-\frac{1}{2}-\frac{1}{2n}}\ln t,
\quad t>3,~\xi\in\mathcal{I},~j=1,2,
\ee
where $C>0$ is a constant independent of $t,\xi,k$. Moreover, as $t\rightarrow\infty$,
\be\label{3.66}
\|(M^{(\pm k_j)})^{-1}(x,t;k)-I\|_{L^\infty(\partial D_\varepsilon(\pm k_j))}=O(t^{-1/2}),
\ee
and
\bea
\frac{1}{2\pi\ii}\int_{\partial D_\varepsilon(k_1)}((M^{(k_1)})^{-1}(x,t;k)-I)\dd k&=&-\frac{(\delta_{k_1}^0)^{\hat{\sigma}_3}M^Y_1(\xi)}
{4\sqrt{tk_1(3\alpha-40\beta k_1^2)}}+O(t^{-1}),\label{3.67}\\
\frac{1}{2\pi\ii}\int_{\partial D_\varepsilon(-k_1)}((M^{(-k_1)})^{-1}(x,t;k)-I)\dd k&=&\frac{(\delta_{-k_1}^0)^{\hat{\sigma}_3}\overline{M^Y_1(\xi)}}
{4\sqrt{tk_1(3\alpha-40\beta k_1^2)}}+O(t^{-1}),\label{3.68}\\
\frac{1}{2\pi\ii}\int_{\partial D_\varepsilon(k_2)}((M^{(k_2)})^{-1}(x,t;k)-I)\dd k&=&-\frac{(\delta_{k_2}^0)^{\hat{\sigma}_3}M^X_1(\xi)}
{4\sqrt{tk_2(40\beta k_2^2-3\alpha)}}+O(t^{-1}),\label{3.69}\\
\frac{1}{2\pi\ii}\int_{\partial D_\varepsilon(-k_2)}((M^{(-k_2)})^{-1}(x,t;k)-I)\dd k&=&\frac{(\delta_{-k_2}^0)^{\hat{\sigma}_3}\overline{M^X_1(\xi)}}
{4\sqrt{tk_2(40\beta k_2^2-3\alpha)}}+O(t^{-1}),\label{3.70}
\eea
where $M^X_1(\xi)$ and $M^Y_1(\xi)$ are given by
\be\label{3.71}
M^X_1(\xi)=-\ii\begin{pmatrix}
0 ~& \beta^X(q)\\[4pt]
\overline{\beta^X(q)} ~& 0
\end{pmatrix},\quad
M^Y_1(\xi)=-\ii\begin{pmatrix}
0 ~& \beta^Y(p)\\[4pt]
\overline{\beta^Y(p)} ~& 0
\end{pmatrix}.
\ee
\end{lemma}
\begin{proof}
We just consider the proof for the function $M^{(k_2)}(x,t;k)$, the others accordingly follow.

The analyticity of $M^{(k_2)}$ is obvious. Since $|\delta_{k_2}^0(z)|=1$, thus, the estimate \eqref{3.64} for $M^{(k_2)}$ follows from the definition of $M^{(k_2)}$ in \eqref{3.59} and the estimate \eqref{3.25}. On the other hand, we have
\berr
J^{(2)}-J^{(k_2)}=(\delta_{k_2}^0)^{\hat{\sigma}_3}(\tilde{J}-J^X),\quad k\in\mathcal{X}_{k_2}^\varepsilon.
\eerr
However, proceeding the similar calculation as the Lemma 3.35 in \cite{PD} (also can see \cite{CVZ,XJ1}), we have
\be\label{3.72}
\|\tilde{J}-J^X\|_{L^\infty((\mathcal{X}_{k_2}^\varepsilon)_1)}\leq C|\e^{\frac{\ii\gamma}{2}z^2}|t^{-1/2}\ln t,\quad 0<\gamma<\frac{1}{2},~t>3,~\xi\in\mathcal{I},
\ee
for $k\in(\mathcal{X}_{k_2}^\varepsilon)_1$, that is, $z=4\sqrt{tk_2(40\beta k_2^2-3\alpha)}s\e^{\frac{\ii\pi}{4}}$, $0\leq s\leq\varepsilon$. Thus,
\be\label{3.73}
\|\tilde{J}-J^X\|_{L^1((\mathcal{X}_{k_2}^\varepsilon)_1)}\leq Ct^{-1}\ln t,\quad t>3,~\xi\in\mathcal{I}.
\ee
By the general inequality $\|f\|_{L^n}\leq\|f\|^{1-1/n}_{L^\infty}\|f\|_{L^1}^{1/n}$, we find
\be
\|\tilde{J}-J^X\|_{L^n((\mathcal{X}_{k_2}^\varepsilon)_1)}\leq Ct^{-1/2-1/2n}\ln t,\quad t>3,~\xi\in\mathcal{I}.
\ee
The norms on $(\mathcal{X}_{k_2}^\varepsilon)_j$, $j=2,3,4$, are estimated in a similar way. Therefore, \eqref{3.65} follows.

If $k\in\partial D_\varepsilon(k_2)$, the variable $z=4\sqrt{tk_2(40\beta k_2^2-3\alpha)}(k-k_2)$ tends to infinity as $t\rightarrow\infty$. It follows from \eqref{2.22} that
\berr
M^X(q,z)=I+\frac{M^X_1(\xi)}{4\sqrt{tk_2(40\beta k_2^2-3\alpha)}(k-k_2)}+O\bigg(\frac{q}{t}\bigg),\quad t\rightarrow\infty,~k\in \partial D_\varepsilon(k_2),
\eerr
where $M^X_1(\xi)$ is defined by \eqref{3.71}. Since
 $$M^{(k_2)}(x,t;k)=(\delta_{k_2}^0)^{\hat{\sigma}_3}M^X(q,z),$$
thus we have
\be\label{3.75}
(M^{(k_2)})^{-1}(x,t;k)-I=-\frac{(\delta_{k_2}^0)^{\hat{\sigma}_3}M^X_1(\xi)}
{4\sqrt{tk_2(40\beta k_2^2-3\alpha)}(k-k_2)}+O\bigg(\frac{q}{t}\bigg),~t\rightarrow\infty,~k\in \partial D_\varepsilon(k_2).
\ee
The estimate \eqref{3.66} immediately follows from \eqref{3.75} and $|M_1^X|\leq C$. By Cauchy's formula and \eqref{3.75}, we derive \eqref{3.69}.
\end{proof}

\subsection{The final step}

We now begin to establish the explicit long-time asymptotic formula for the emKdV equation \eqref{1.1} on the line.
Define the approximate solution $M^{(app)}(x,t;k)$ by
\be\label{3.76}
M^{(app)}=\left\{\begin{aligned}
&M^{(-k_2)},\quad k\in D_\varepsilon(-k_2),\\
&M^{(-k_1)},\quad k\in D_\varepsilon(-k_1),\\
&M^{(k_1)},~~\quad k\in D_\varepsilon(k_1),\\
&M^{(k_2)},~~\quad k\in D_\varepsilon(k_2),\\
&I,\qquad ~~~~~{\text elsewhere}.
\end{aligned}
\right.
\ee
Let $\hat{M}(x,t;k)$ be
\be\label{3.77}
\hat{M}=M^{(2)}(M^{(app)})^{-1},
\ee
 then $\hat{M}(x,t;k)$ satisfies the following RH problem
\be\label{3.78}
\hat{M}_+(x,t;k)=\hat{M}_-(x,t;k)\hat{J}(x,t;k),\quad k\in\hat{\Sigma},
\ee
where the jump contour $\hat{\Sigma}=\Sigma\cup\partial D_\varepsilon(-k_2)\cup\partial D_\varepsilon(-k_1)\cup\partial D_\varepsilon(k_1)\cup\partial D_\varepsilon(k_2)$ is depicted in Fig. \ref{fig6}, and the jump matrix $\hat{J}(x,t;k)$ is given by
\be\label{3.79}
\hat{J}=\left\{
\begin{aligned}
&M^{(app)}_-J^{(2)}(M^{(app)}_+)^{-1},~~ k\in\hat{\Sigma}\cap (D_\varepsilon(-k_2)\cup D_\varepsilon(-k_1)\cup D_\varepsilon(k_1)\cup D_\varepsilon(k_2)),\\
&(M^{(app)})^{-1},\qquad\qquad\quad k\in(\partial D_\varepsilon(-k_2)\cup\partial D_\varepsilon(-k_1)\cup\partial D_\varepsilon(k_1)\cup\partial D_\varepsilon(k_2)),\\
&J^{(2)},\qquad\qquad\qquad\quad~~~ k\in\hat{\Sigma}\setminus (\overline{D_\varepsilon(-k_2)}\cup\overline{D_\varepsilon(-k_1)}\cup\overline{D_\varepsilon(k_1)}\cup\overline{D_\varepsilon(k_2)}).
\end{aligned}
\right.
\ee
\begin{figure}[htbp]
\centering
\includegraphics[width=4in]{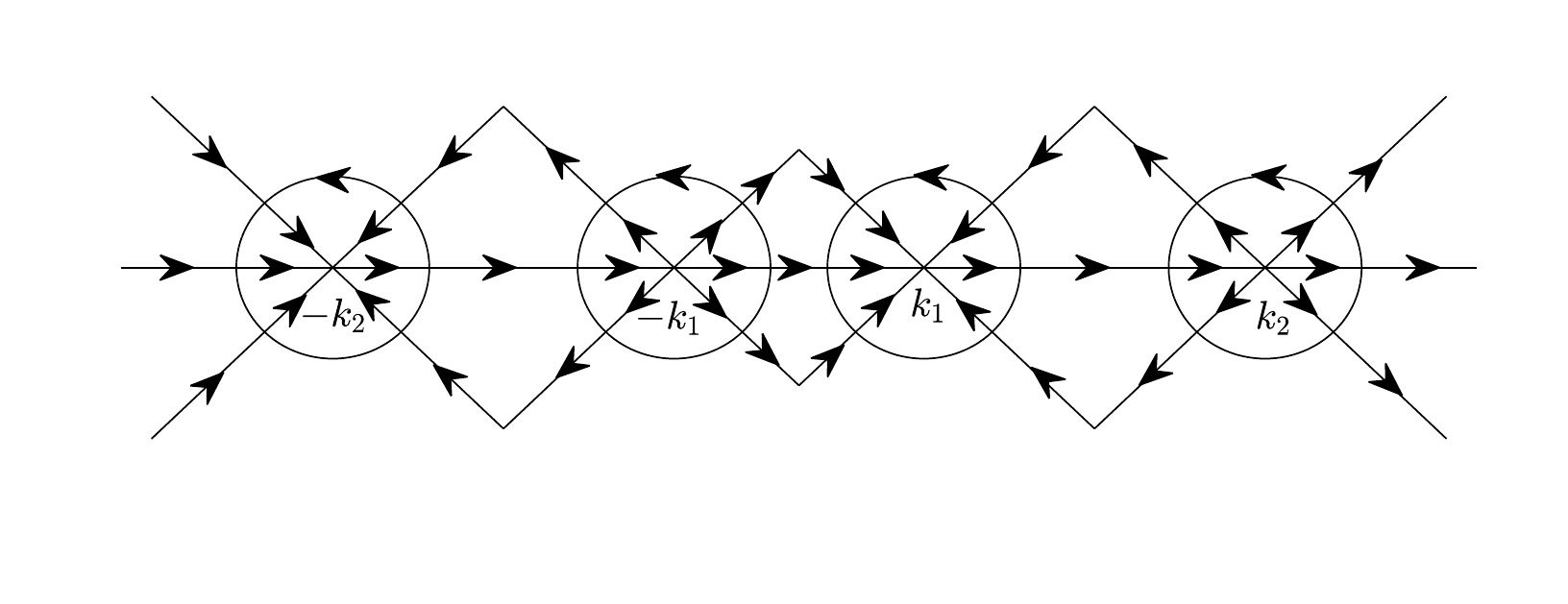}
\caption{The contour $\hat{\Sigma}$.}\label{fig6}
\end{figure}

For convenience, we rewrite $\hat{\Sigma}$ as follows:
$$\hat{\Sigma}=\hat{\Sigma}_1\cup\hat{\Sigma}_2\cup\hat{\Sigma}_3\cup\hat{\Sigma}_4,$$
where
\bea
\hat{\Sigma}_1&=&\bigcup_1^4\Sigma_j\setminus(D_\varepsilon(-k_2)\cup D_\varepsilon(-k_1)\cup D_\varepsilon(k_1)\cup D_\varepsilon(k_2)),~\hat{\Sigma}_2=\bigcup_5^{6}\Sigma_j,\nn\\
\hat{\Sigma}_3&=&\partial D_\varepsilon(-k_2)\cup\partial D_\varepsilon(-k_1)\cup\partial D_\varepsilon(k_1)\cup\partial D_\varepsilon(k_2),\nn\\
\hat{\Sigma}_4&=&\mathcal{X}_{-k_2}^\varepsilon\cup\mathcal{X}_{-k_1}^\varepsilon
\cup\mathcal{X}_{k_1}^\varepsilon\cup\mathcal{X}_{k_2}^\varepsilon.\nn
\eea
and $\{\Sigma_j\}_1^{6}$ denoting the restriction of $\Sigma$ to the contour labeled by $j$ in Fig. \ref{fig5}.
 Then we have the following lemma if we let $\hat{w}=\hat{J}-I$.
\begin{lemma}\label{lem4}
For $1\leq n\leq\infty$, $t>3$ and $\xi\in\mathcal{I}$, the following estimates hold:
\bea
\|\hat{w}\|_{L^n(\hat{\Sigma}_1)}&\leq&C\e^{-ct},\label{3.80}\\
\|\hat{w}\|_{L^n(\hat{\Sigma}_2)}&\leq&Ct^{-3/2},\label{3.81}\\
\|\hat{w}\|_{L^n(\hat{\Sigma}_3)}&\leq& Ct^{-1/2},\label{3.82}\\
\|\hat{w}\|_{L^n(\hat{\Sigma}_4)}&\leq&Ct^{-\frac{1}{2}-\frac{1}{2n}}\ln t.\label{3.83}
\eea
\end{lemma}
\begin{proof}
For $k\in\Omega_1\cap\{k\in\bfC|\text{Re}k>k_2\}\cap\hat{\Sigma}_1$, we have $-|\text{Re}\Phi(k)|\leq-c\varepsilon^2.$
Since $\hat{w}$ only has a nonzero $r_{1,a}\delta^{-2}\e^{t\Phi}$ in $(21)$ entry, hence, for $t\geq1$, by \eqref{3.15}, we get
\bea
|\hat{w}_{21}|&=&|r_{1,a}\delta^{-2}\e^{t\Phi}|
\leq\frac{C}{1+|k|^2}\e^{-\frac{3t}{4}|\text{Re}\Phi|}\leq  C\e^{-c\varepsilon^2t}.\nn
\eea
In a similar way, the other estimates on $\hat{\Sigma}_1$ hold. This proves \eqref{3.80}.
Since the matrix $\hat{w}$ on $\hat{\Sigma}_2$ only involves the small remainders $r_{j,r}$ for $j=1,\cdots,4$, by Lemma \ref{lem2}, the estimate \eqref{3.81} follows.
The inequality \eqref{3.82} is a consequence of \eqref{3.66}, \eqref{3.76} and \eqref{3.79}.
For $k\in\mathcal{X}_{\pm k_j}^\varepsilon$, we find $$\hat{w}=M^{(\pm k_j)}_-(J^{(2)}-J^{(\pm k_j)})(M^{(\pm k_j)}_+)^{-1},\quad j=1,2.$$ Therefore, it follows from \eqref{3.64} and \eqref{3.65} that the estimate \eqref{3.83} holds.
\end{proof}
The estimates in Lemma \ref{lem4} imply that
\be\label{3.84}
\begin{aligned}
\|\hat{w}\|_{(L^1\cap L^2)(\hat{\Sigma})}&\leq Ct^{-1/2},\\
\|\hat{w}\|_{L^\infty(\hat{\Sigma})}&\leq Ct^{-1/2}\ln t,
\end{aligned}
\quad t>3,~ \xi\in\mathcal{I}.
\ee
Let $\hat{C}$ denote the Cauchy operator associated with $\hat{\Sigma}$:
\berr
(\hat{C}f)(k)=\int_{\hat{\Sigma}}\frac{f(\zeta)}{\zeta-k}\frac{\dd \zeta}{2\pi\ii},\quad k\in\bfC\setminus\hat{\Sigma},~f\in L^2(\hat{\Sigma}).
\eerr
We denote the boundary values of $\hat{C}f$ from the left and right sides of $\hat{\Sigma}$ by $\hat{C}_+f$ and $\hat{C}_-f$, respectively. Define the operator $\hat{C}_{\hat{w}}$: $L^2(\hat{\Sigma})+L^\infty(\hat{\Sigma})\rightarrow L^2(\hat{\Sigma})$ by $\hat{C}_{\hat{w}}f=\hat{C}_-(f\hat{w}),$ that is, $\hat{C}_{\hat{w}}$ is defined by $\hat{C}_{\hat{w}}(f)=\hat{C}_+(f\hat{w}_-)+\hat{C}_-(f\hat{w}_+)$ where we have chosen, for simplicity, $\hat{w}_+=\hat{w}$ and $\hat{w}_-=0$. Then, by \eqref{3.84}, we find
\be\label{3.85}
\|\hat{C}_{\hat{w}}\|_{B(L^2(\hat{\Sigma}))}\leq C\|\hat{w}\|_{L^\infty(\hat{\Sigma})}\leq Ct^{-1/2}\ln t,
\ee
where $B(L^2(\hat{\Sigma}))$ denotes the Banach space of bounded linear operators $L^2(\hat{\Sigma})\rightarrow L^2(\hat{\Sigma})$. Therefore, there exists a $T>0$ such that $I-\hat{C}_{\hat{w}}\in B(L^2(\hat{\Sigma}))$ is invertible for all $\xi\in\mathcal{I},$ $t>T$. Following this, we may define the $2\times2$ matrix-valued function $\hat{\mu}(x,t;k)$ whenever $t>T$ by
\be\label{3.86}
\hat{\mu}=I+\hat{C}_{\hat{w}}\hat{\mu}.
\ee
Then
\be\label{3.87}
\hat{M}(x,t;k)=I+\frac{1}{2\pi\ii}\int_{\hat{\Sigma}}\frac{(\hat{\mu}\hat{w})(x,t;\zeta)}{\zeta-k}\dd\zeta,\quad k\in\bfC\setminus\hat{\Sigma}
\ee
is the unique solution of the RH problem \eqref{3.78} for $t>T$. Moreover, using the Neumann series (see \cite{JL3}), the function $\hat{\mu}(x,t;k)$ satisfies
\be\label{3.88}
\|\hat{\mu}(x,t;\cdot)-I\|_{L^2(\hat{\Sigma})}=O(t^{-1/2}),\quad t\rightarrow\infty,~\xi\in\mathcal{I}.
\ee
It follows from \eqref{3.87} that
\be\label{3.89}
\lim_{k\rightarrow\infty}k(\hat{M}(x,t;k)-I)=-\frac{1}{2\pi\ii}\int_{\hat{\Sigma}}(\hat{\mu}\hat{w})(x,t;k)\dd k.
\ee
Using \eqref{3.80} and \eqref{3.88}, we have
\bea
\int_{\hat{\Sigma}_1}(\hat{\mu}\hat{w})(x,t;k)\dd k&=&\int_{\hat{\Sigma}_1}\hat{w}(x,t;k)\dd k+\int_{\hat{\Sigma}_1}(\hat{\mu}(x,t;k)-I)\hat{w}(x,t;k)\dd k\nn\\
&\leq&\|\hat{w}\|_{L^1(\hat{\Sigma}_1)}+\|\hat{\mu}-I\|_{L^2(\hat{\Sigma}_1)}\|\hat{w}\|_{L^2(\hat{\Sigma}_1)}\nn\\
&\leq&C\e^{-ct},\quad t\rightarrow\infty.\nn
\eea
By \eqref{3.81} and \eqref{3.88}, the contribution from $\hat{\Sigma}_2$ to the right-hand side of \eqref{3.89} is
\berr
O(\|\hat{w}\|_{L^1(\hat{\Sigma}_2)}+\|\hat{\mu}-I\|_{L^2(\hat{\Sigma}_2)}
\|\hat{w}\|_{L^2(\hat{\Sigma}_2)})=O(t^{-3/2}),\quad t\rightarrow\infty.
\eerr
Similarly, by \eqref{3.83} and \eqref{3.88}, the contribution from $\hat{\Sigma}_4$ to the right-hand side of \eqref{3.89} is $$O(\|\hat{w}\|_{L^1(\hat{\Sigma}_4)}+\|\hat{\mu}-I\|_{L^2(\hat{\Sigma}_4)}
\|\hat{w}\|_{L^2(\hat{\Sigma}_4)})=O(t^{-1}\ln t),\quad t\rightarrow\infty.$$
Finally, by \eqref{3.67}-\eqref{3.70}, \eqref{3.82} and \eqref{3.88}, we can get
\bea
&&~~-\frac{1}{2\pi\ii}\int_{\hat{\Sigma}_3}(\hat{\mu}\hat{w})(x,t;k)\dd k\nn\\
&&\quad=-\frac{1}{2\pi\ii}\int_{\hat{\Sigma}_3}\hat{w}(x,t;k)\dd k-\frac{1}{2\pi\ii}\int_{\hat{\Sigma}_3}(\hat{\mu}(x,t;k)-I)\hat{w}(x,t;k)\dd k\nn\\
&&\quad=-\frac{1}{2\pi\ii}\int_{\partial D_\varepsilon(k_1)}\bigg((M^{(k_1)})^{-1}(x,t;k)-I\bigg)\dd k-\frac{1}{2\pi\ii}\int_{\partial D_\varepsilon(-k_1)}\bigg((M^{(-k_1)})^{-1}(x,t;k)-I\bigg)\dd k\nn\\
&&\qquad-\frac{1}{2\pi\ii}\int_{\partial D_\varepsilon(k_2)}\bigg((M^{(k_2)})^{-1}(x,t;k)-I\bigg)\dd k-\frac{1}{2\pi\ii}\int_{\partial D_\varepsilon(-k_2)}\bigg((M^{(-k_2)})^{-1}(x,t;k)-I\bigg)\dd k\nn\\
&&\qquad+O(\|\hat{\mu}-I\|_{L^2(\hat{\Sigma}_3)}
\|\hat{w}\|_{L^2(\hat{\Sigma}_3)})\nn\\
&&\quad=\frac{(\delta_{k_1}^0)^{\hat{\sigma}_3}M^Y_1(\xi)}
{4\sqrt{tk_1(3\alpha-40\beta k_1^2)}}-\frac{(\delta_{-k_1}^0)^{\hat{\sigma}_3}\overline{M^Y_1(\xi)}}
{4\sqrt{tk_1(3\alpha-40\beta k_1^2)}}\nn\\
&&~\qquad+\frac{(\delta_{k_2}^0)^{\hat{\sigma}_3}M^X_1(\xi)}
{4\sqrt{tk_2(40\beta k_2^2-3\alpha)}}-\frac{(\delta_{-k_2}^0)^{\hat{\sigma}_3}\overline{M^X_1(\xi)}}
{4\sqrt{tk_2(40\beta k_2^2-3\alpha)}}+O(t^{-1}),\quad t\rightarrow\infty.\nn
\eea
Thus, we obtain the following important relation
\bea\label{3.90}
\begin{aligned}
\lim_{k\rightarrow\infty}k(\hat{M}(x,t;k)-I)&=\frac{(\delta_{k_1}^0)^{\hat{\sigma}_3}M^Y_1(\xi)}
{4\sqrt{tk_1(3\alpha-40\beta k_1^2)}}-\frac{(\delta_{-k_1}^0)^{\hat{\sigma}_3}\overline{M^Y_1(\xi)}}
{4\sqrt{tk_1(3\alpha-40\beta k_1^2)}}\\
&+\frac{(\delta_{k_2}^0)^{\hat{\sigma}_3}M^X_1(\xi)}
{4\sqrt{tk_2(40\beta k_2^2-3\alpha)}}-\frac{(\delta_{-k_2}^0)^{\hat{\sigma}_3}\overline{M^X_1(\xi)}}
{4\sqrt{tk_2(40\beta k_2^2-3\alpha)}}+O(t^{-1}\ln t),\quad t\rightarrow\infty.
\end{aligned}
\eea
Taking into account that \eqref{3.4}, \eqref{3.8}, \eqref{3.37} and \eqref{3.77}, for sufficient large $k\in\bfC\setminus\hat{\Sigma}$, we get
\bea\label{3.91}
u(x,t)&=&-2\ii\lim_{k\rightarrow\infty}(kM(x,t;k))_{12}\nn\\
&=&-2\ii\lim_{k\rightarrow\infty}k(\hat{M}(x,t;k)-I)_{12}\nn\\
&=&-\bigg(\frac{\beta^Y(p)(\delta_{k_1}^0)^{2}+\overline{\beta^Y(p)}(\delta_{-k_1}^0)^{2}}
{2\sqrt{tk_1(3\alpha-40\beta k_1^2)}}+\frac{\beta^X(q)(\delta_{k_2}^0)^2+\overline{\beta^X(q)}(\delta_{-k_2}^0)^2}{2\sqrt{tk_2(40\beta k_2^2-3\alpha)}}\bigg)+O\bigg(\frac{\ln t}{t}\bigg).\nn
\eea
Using
\be
\delta_{-k_1}^0=\overline{\delta_{k_1}^0},\quad\delta_{-k_2}^0=\overline{\delta_{k_2}^0}
\ee
as $\chi_j(-k_j)=-\chi_j(k_j)=\overline{\chi_j(k_j)},~j=1,2$, and collecting the above computations, we obtain our main results as stated in the Theorem \ref{the4.2}.
\section{Asymptotics for a special case $\alpha=0$}
\setcounter{equation}{0}
\setcounter{lemma}{0}
\setcounter{theorem}{0}
In this section, we consider the long-time asymptotics to the solutions for a particular case, namely, $\alpha=0$ of emKdV equation \eqref{1.1},
\begin{equation}\label{4.1}
u_t+\beta(30u^4u_x+10u_x^3+40uu_xu_{xx}+10u^2u_{xxx}+u_{xxxxx})=0.
\end{equation}
As we all know, the Hirota equation can be reduced to the complex-valued mKdV equation under the Galilean transformation. Thus, if we rewrite equation \eqref{1.1} as a complex-valued form
\be\label{4.2}
u_t+\alpha(6|u|^2u_x+u_{xxx})+\beta[30|u|^4u_x+10(u|u_x|^2+|u|^2u_{xx})_x+u_{xxxxx}]=0,
\ee
similarly, the Galilean transformation can reduce \eqref{4.2} into a complex-valued form of \eqref{4.1}, however, where we take $u$ is a real-valued function. In fact, the fifth order KdV equation indeed can be excluded the third derivative term, (see \cite{DG}), where its prolongation structure was considered.

As in Section 2, under the condition that $a(k)\neq0$ for $\{k\in\bfC|\text{Im}k\geq0\}$, the corresponding RH problem associated with \eqref{4.1} is
\be\label{4.3}
\left\{
\begin{aligned}
&N_+(x,t;k)=N_-(x,t;k)v(x,t;k),~k\in \bfR,\\
&N(x,t;k)\rightarrow I,~\qquad\qquad\qquad\qquad~k\rightarrow\infty,
\end{aligned}
\right.
\ee
where the jump matrix $v(x,t;k)$ is defined by
\be\label{4.4}
\begin{aligned}
&v(x,t;k)=\begin{pmatrix}
1+|r(k)|^2 ~& \bar{r}(k)\e^{-t\theta(k)}\\[4pt]
r(k)\e^{t\theta(k)} ~& 1
\end{pmatrix},\\
&r(k)=\frac{\bar{b}(k)}{a(k)},~ \theta(k)=2\ii(16\beta k^5-k\xi),~ \xi=\frac{x}{t}.
\end{aligned}
\ee
Also we have
\be\label{4.5}
r(-k)=\overline{r(\bar{k})},\quad k\in\bfR.
\ee
Moreover, the relation between the solution $u(x,t)$ of the equation \eqref{4.1} and $N(x,t;k)$ is
\be\label{4.6}
u(x,t)=-2\ii\lim_{k\rightarrow\infty}(kN(x,t;k))_{12}.
\ee
For this case, there are two real and pure imaginary critical points of $\theta(k)$ located at the points $\pm k_0$ and $\pm\ii k_0$, where
\be\label{4.7}
k_0=\sqrt[4]{\frac{\xi}{80\beta}}.
\ee

Our aim in this section is to find the asymptotics of solution $u(x,t)$ to the equation \eqref{4.1} in region $0<x\leq Mt^{\frac{1}{5}}$, where $M>1$ is a constant. We see that as $t\rightarrow\infty$, the critical points $\pm k_0$ approach 0 at least as fast as $t^{-\frac{1}{5}}$, i.e., $k_0\leq\big(\frac{M}{80\beta}\big)^{\frac{1}{4}}t^{-\frac{1}{5}}$. We will show that the asymptotics of the solution $u(x,t)$ in this region is given in terms of the solution of a fourth order Painlev\'e II equation.

Let $\Gamma\subset\bfC$ denote the contour $\Gamma=\bfR\cup\Gamma_1\cup\Gamma_2$, where
\be
\begin{aligned}
\Gamma_1&=\{k_0+l\e^{\frac{\pi\ii}{6}}|l\geq0\}\cup\{-k_0+l\e^{\frac{5\pi\ii}{6}}|l\geq0\},\\
\Gamma_2&=\{k_0+l\e^{-\frac{\pi\ii}{6}}|l\geq0\}\cup\{-k_0+l\e^{-\frac{5\pi\ii}{6}}|l\geq0\},
\end{aligned}
\ee
and we orient $\Gamma$ to the right. Let $\mathcal{V}$ and $\mathcal{V}^*$ denote the triangular domains shown in Fig. \ref{fig7}.
\begin{figure}[htbp]
\centering
\includegraphics[width=3.5in]{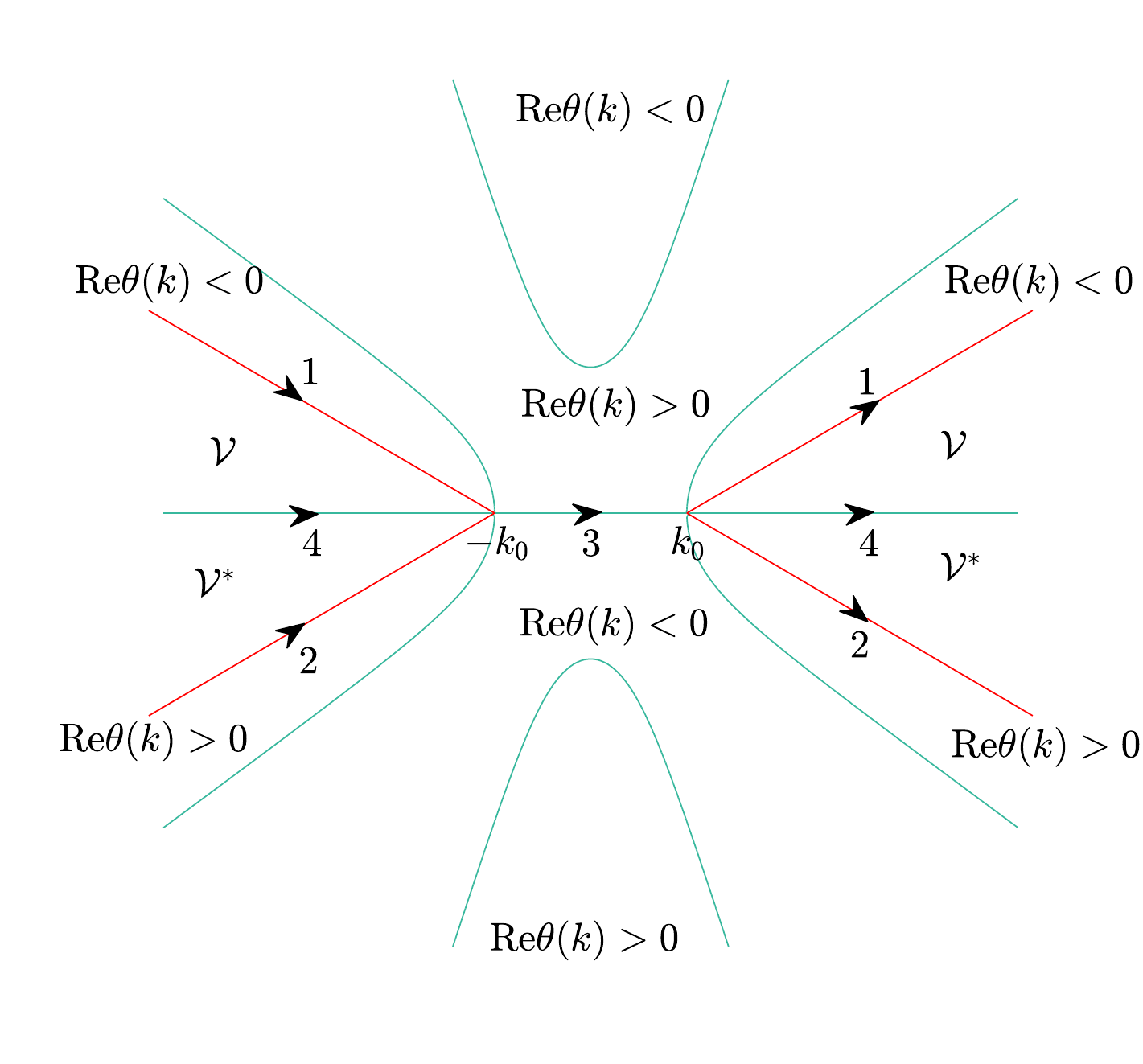}
\caption{The oriented contour $\Gamma$ and the sets $\mathcal{V}$ and $\mathcal{V}^*$.}\label{fig7}
\end{figure}

Denote $$\mathcal{P}=\{(x,t)\in\bfR^2|0<x\leq Mt^{\frac{1}{5}},t\geq3\}.$$ By Lemma \ref{lem2}, we also have the following analytic approximation lemma for $r(k)$.
\begin{lemma}\label{lemma4.1}
There exists a decomposition
\be
r(k)=r_a(x,t,k)+r_r(x,t,k), \quad k\in(-\infty,-k_0)\cup(k_0,\infty),
\ee
where the functions $r_a$ and $r_r$ satisfy the following properties:\\
(i) For $(x,t)\in\mathcal{P}$, $r_{a}(x,t,k)$ is defined and continuous for $k\in\bar{\mathcal{V}}$ and analytic for $\mathcal{V}$.\\
(ii) The function $r_{a}$ satisfies
\be\label{4.10}
|r_{a}(x,t,k)|\leq \frac{C}{1+|k|^2}\e^{\frac{t}{4}|\text{Re}\theta(k)|},~k\in\bar{\mathcal{V}},
\ee
and
\be
|r_{a}(x,t,k)-r(k_0)|\leq C|k-k_0|\e^{\frac{t}{4}|\text{Re}\theta(k)|},~k\in\bar{\mathcal{V}}.
\ee
(iii) The $L^1, L^2$ and $L^\infty$ norms of the function $r_{r}(x,t,\cdot)$ on $(-\infty,-k_0)\cup(k_0,\infty)$ are $O(t^{-3/2})$ as $t\rightarrow\infty$ uniformly with respect to $(x,t)\in\mathcal{P}$.\\
(iv) The following symmetries hold:
\be\label{4.12}
r_{a}(x,t,k)=\overline{r_{a}(x,t,-\bar{k})},~r_{r}(x,t,k)=\overline{r_{r}(x,t,-\bar{k})}.
\ee
\end{lemma}
The first transform is as follows:
\be\label{4.13}
N^{(1)}(x,t;k)=N(x,t;k)\times\left\{
\begin{aligned}
&\begin{pmatrix}
1 & 0\\[4pt]
-r_a(x,t,k)\e^{t\theta(k)} & 1
\end{pmatrix},~~~~k\in \mathcal{V},\\
&\begin{pmatrix}
1 &~ \overline{r_a(x,t,\bar{k})}\e^{-t\theta(k)}\\[4pt]
0 & 1
\end{pmatrix},~~~~k\in \mathcal{V}^*,\\
&I,\qquad\qquad\qquad\qquad~~~~~~~~ \text{elsewhere}.
\end{aligned}
\right.
\ee
Then we obtain the RH problem
\be\label{4.14}
N^{(1)}_+(x,t;k)=N^{(1)}_-(x,t;k)v^{(1)}(x,t,k)
\ee
on the contour $\Gamma$ depicted in Fig. \ref{fig7}. The jump matrix $v^{(1)}(x,t,k)$ is given by
\bea
v^{(1)}_1&=&\begin{pmatrix}
1 &~ 0\\[4pt]
r_a\e^{t\theta} &~ 1
\end{pmatrix},\qquad\qquad\qquad v^{(1)}_2=\begin{pmatrix}
1 &~ \bar{r}_a\e^{-t\theta}\\[4pt]
0 &~ 1
\end{pmatrix},\nn\\
v^{(1)}_3&=&\begin{pmatrix}
1 &~ \bar{r}\e^{-t\theta}\\[4pt]
0 &~ 1
\end{pmatrix}\begin{pmatrix}
1 &~ 0\\[4pt]
r\e^{t\theta} &~ 1
\end{pmatrix},~~~~v^{(1)}_4=\begin{pmatrix}
1 ~& \bar{r}_r\e^{-t\theta} \\[4pt]
0 ~& 1 \\
\end{pmatrix}\begin{pmatrix}
1 &~ 0\\[4pt]
r_r\e^{t\theta} &~ 1
\end{pmatrix},\nn
\eea
where $v^{(1)}_i$ denotes the restriction of $v^{(1)}$ to the contour labeled by $i$ in Fig. \ref{fig7}.

Let us introduce the new variables $y$ and $z$ by
\be
y=\frac{-x}{(20\beta t)^{\frac{1}{5}}},\quad z=(20\beta t)^{\frac{1}{5}}k,
\ee
such that
\be
t\theta(k)=2\ii\bigg(\frac{4}{5}z^5+yz\bigg).
\ee
We now have $-C\leq y<0$. Fix $\varepsilon>0$ and let $D_\varepsilon(0)=\{k\in\bfC||k|<\varepsilon\}$ denote the open disk of radius $\varepsilon$ centered at the origin. Let $\Gamma^\varepsilon=(\Gamma\cap D_\varepsilon(0))\setminus((-\infty,-k_0)\cup(k_0,\infty))$. Let $Z$ denote the contour defined in (B.1) with $z_0=(20\beta t)^{\frac{1}{5}}k_0=\sqrt[4]{-y}/\sqrt{2}$. The map $k\mapsto z$ maps $\Gamma^\varepsilon$ onto $Z\cap\{z\in\bfC||z|<(20\beta t)^{\frac{1}{5}}\varepsilon\}$. We write $\Gamma^\varepsilon=\cup_{j=1}^3\Gamma_j^\varepsilon$, where $\Gamma_j^\varepsilon$ denotes the inverse image of $Z_j\cap\{z\in\bfC||z|<(20\beta t)^{\frac{1}{5}}\varepsilon\}$ under this map, see Fig. \ref{fig8}.
\begin{figure}[htbp]
\centering
\includegraphics[width=3in]{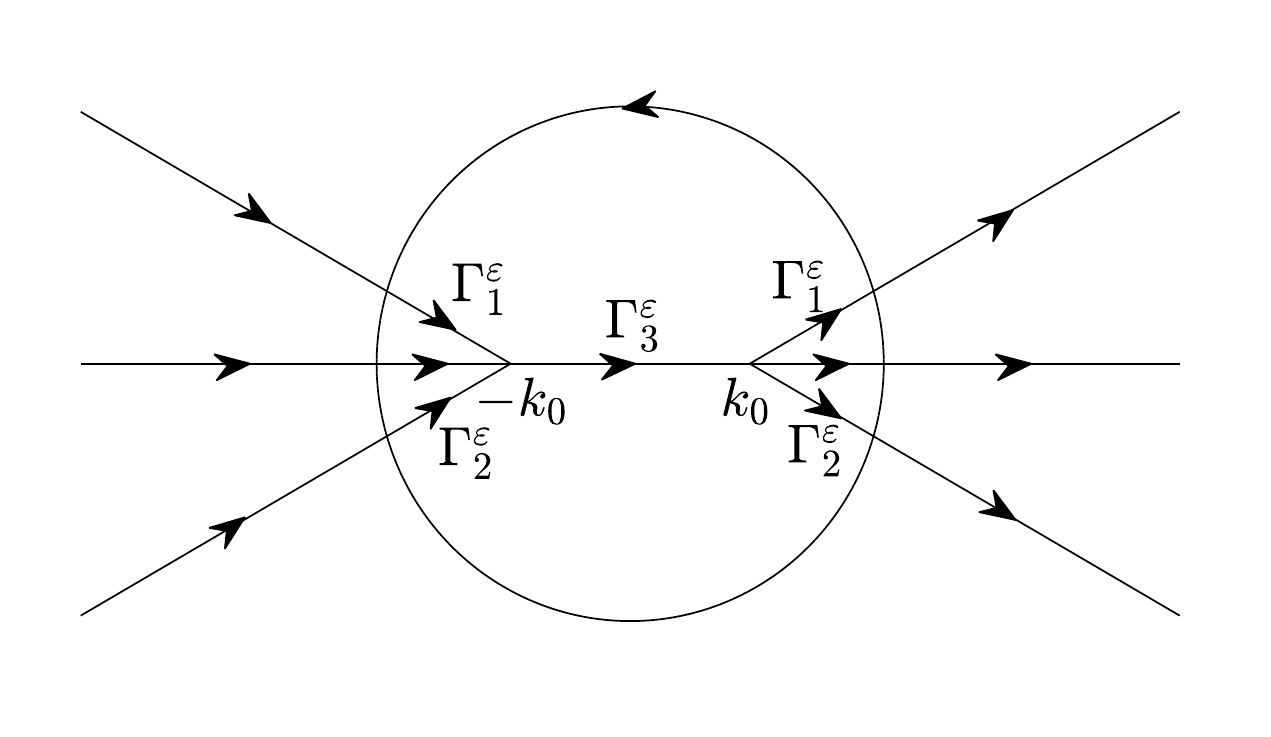}
\caption{The oriented contour $\hat{\Gamma}$ and $\Gamma^\varepsilon$.}\label{fig8}
\end{figure}

For large $t$ and fixed $z$, the jump matrices $\{v^{(1)}_j\}_1^4$ tend to the matrix $v^Z$ defined in (B.3) if we set $s=r(0)$.
Thus we expect that $N^{(1)}$ approaches the solution $N^0(x,t,k)$ defined by
\be\label{4.17}
N^0(x,t;k):=N^Z(s,y,z,z_0)
\ee
for large $t$, where $N^Z(s,y,z,z_0)$ is the solution of the model RH problem for (B.2) with $z_0=\sqrt[4]{-y}/\sqrt{2}$. Moreover, if $(x,t)\in\mathcal{P}$, then $(y,t,z_0)\in\Bbb P$, where $\Bbb P$ is the parameter subset defined in (B.4). Thus, Lemma B.1 ensures that $N^0$ is well-defined by \eqref{4.17}.
\begin{lemma}
For each $(x,t)\in\mathcal{P}$, the function $N^0(x,t;k)$ defined in \eqref{4.17} is an analytic function of $k\in D_\varepsilon(0)\setminus\Gamma^\varepsilon$ such that
\be\label{4.18}
|N^0(x,t;k)|\leq C,\quad(x,t)\in\mathcal{P},~k\in D_\varepsilon(0)\setminus\Gamma^\varepsilon.
\ee
Across $\Gamma^\varepsilon$, $N^0$ obeys the jump condition $N_+^0=N_-^0v^0$, where the jump matrix $v^0$ satisfies, for $1\leq n\leq\infty$,
\be\label{4.19}
\|v^{(1)}-v^0\|_{L^n(\Gamma^\varepsilon)}\leq Ct^{-\frac{1}{5}(1+\frac{1}{n})}, \quad(x,t)\in\mathcal{P}.
\ee
Furthermore, as $t\rightarrow\infty$, we have
\be\label{4.20}
\|N^0(x,t;k)^{-1}-I\|_{L^\infty(\partial D_\varepsilon(0))}=O(t^{-\frac{1}{5}}),
\ee
and
\be\label{4.21}
\frac{1}{2\pi\ii}\int_{\partial D_\varepsilon(0)}(N^0(x,t;k)^{-1}-I)\dd k=-\frac{\ii N_1^0(y)}{(20\beta t)^{\frac{1}{5}}},
\ee
where
\be\label{4.22}
 N_1^0(y)=\begin{pmatrix}
-2\int^y_{-\infty}u_p^2(y')\dd y' ~& u_p(y)\\
u_p(y) ~& 2\int^y_{-\infty}u_p^2(y')\dd y'
\end{pmatrix}.
\ee
\end{lemma}
\begin{proof}
The analyticity and boundedness of $N^0$ are a consequence of Lemma B.1. Moreover,
\be
v^{(1)}-v^0=\left\{
\begin{aligned}
&\begin{pmatrix}
0 &~ 0\\
(r_a(x,t,k)-r(0))\e^{t\theta} &~ 0
\end{pmatrix},\qquad\qquad\quad k\in\Gamma^\varepsilon_1,\\
&\begin{pmatrix}
0 &~ (\overline{r_a(x,t,\bar{k}})-\overline{r(0)})\e^{-t\theta}\\
0 &~ 0
\end{pmatrix},~~\quad\qquad\quad k\in\Gamma^\varepsilon_2,\\
&\begin{pmatrix}
|r(k)|^2-|r(0)|^2 &~ (\overline{r(k)}-\overline{r(0)})\e^{-t\theta}\\
(r(k)-r(0))\e^{t\theta} &~ 0
\end{pmatrix},~~ k\in\Gamma^\varepsilon_3.
\end{aligned}
\right.
\ee
For $k=k_0+l\e^{\frac{\pi\ii}{6}}$, $0\leq l\leq\varepsilon$, we obtain
\berr
\text{Re}\theta(k)=-16\beta l^2(l^3+5\sqrt{3}k_0l^2+20k_0^2l+10\sqrt{3}k_0^3)\leq-16\beta|k-k_0|^5.
\eerr
On the other hand, if $|k-k_0|\geq k_0$, then $|k-k_0|\geq|k|/2$, and hence
\berr
\e^{-12\beta t|k-k_0|^5}\leq\e^{-\frac{3}{8}\beta t|k|^5}.
\eerr
If $|k-k_0|<k_0$, then $|k|\leq Ct^{-\frac{1}{5}}$, and so
\berr
\e^{-12\beta t|k-k_0|^5}\leq 1\leq C\e^{-\frac{3}{8}\beta t|k|^5}.
\eerr
Thus, for $k=k_0+l\e^{\frac{\pi\ii}{6}}$, $0\leq l\leq\varepsilon$, we find
\be
\e^{-\frac{3}{4}t|\text{Re}\theta|}\leq\e^{-12\beta t|k-k_0|^5}\leq C\e^{-\frac{3}{8}\beta t|k|^5}\leq C\e^{-\frac{3}{160}|z|^5}.
\ee
As a consequence, we have
\be\label{4.25}
\begin{aligned}
|v^{(1)}-v^0|&\leq C|r_a(x,t,k)-r(k_0)|\e^{t\text{Re}\theta}+C|r(k_0)-r(0)|\e^{t\text{Re}\theta}\\
&\leq C|k-k_0|\e^{-\frac{3}{4}t|\text{Re}\theta|}+Ck_0\e^{-t|\text{Re}\theta|}\leq C|zt^{-\frac{1}{5}}|\e^{-\frac{3}{160}|z|^5}.
\end{aligned}
\ee
A similar computation shows that \eqref{4.25} also holds for $k=-k_0+l\e^{\frac{5\pi\ii}{6}}$, $0\leq l\leq\varepsilon$. Consequently, writing $l=|z|$,
\be
\|v^{(1)}-v^0\|_{L^\infty(\Gamma^\varepsilon_1)}\leq Ct^{-\frac{1}{5}},
\ee
and
\be
\|v^{(1)}-v^0\|_{L^1(\Gamma^\varepsilon_1)}\leq C\int_0^\infty lt^{-\frac{1}{5}}\e^{-\frac{3}{160}l^5}\frac{\dd l}{t^{\frac{1}{5}}}\leq Ct^{-\frac{2}{5}}.
\ee
Using the general inequality $\|f\|_{L^n}\leq\|f\|^{1-1/n}_{L^\infty}\|f\|_{L^1}^{1/n}$, \eqref{4.19} holds for $k\in\Gamma^\varepsilon_1$. Similar estimates applying to $\Gamma^\varepsilon_j$, $j=2,3$ show that \eqref{4.19} holds.

The variable $z=(20\beta t)^{\frac{1}{5}}k$ satisfies $|z|=(20\beta t)^{\frac{1}{5}}\varepsilon$ if $|k|=\varepsilon$. Thus, equation (B.5) yields
\be\label{4.28}
N^0(x,t;k)=I+\frac{\ii N^0_1(y)}{(20\beta t)^{\frac{1}{5}}k}+O(t^{-\frac{2}{5}}),\quad k\in\partial D_\varepsilon(0),~t\rightarrow\infty.
\ee
Thus, \eqref{4.20} and \eqref{4.21} follow from \eqref{4.28} and Cauchy's formula.
\end{proof}

Let $\hat{\Gamma}=\Gamma\cup\partial D_\varepsilon(0)$ and assume that the boundary of $D_\varepsilon(0)$ is oriented counterclockwise, see Fig. \ref{fig8}. Define $\hat{N}(x,t;k)$ by
\be
\hat{N}(x,t;k)=\left\{
\begin{aligned}
&N^{(1)}(x,t;k)N^0(x,t;k)^{-1},\quad k\in D_\varepsilon(0),\\
&N^{(1)}(x,t;k),~\qquad\qquad\qquad k\in\bfC\setminus D_\varepsilon(0),
\end{aligned}
\right.
\ee
then $\hat{N}(x,t;k)$ satisfies the following RH problem
\be\label{4.30}
\hat{N}_+(x,t;k)=\hat{N}_-(x,t;k)\hat{v}(x,t;k),\quad k\in\hat{\Gamma},
\ee
where the jump contour $\hat{\Gamma}=\Gamma^\varepsilon\cup\partial D_\varepsilon(0)\cup(\bfR\setminus[-k_0,k_0])\cup\hat{\Gamma}'$, $\hat{\Gamma}'=\Gamma\setminus(\bfR\cup\overline{D_\varepsilon(0)})$ is depicted in Fig. \ref{fig8}, and the jump matrix $\hat{v}(x,t;k)$ is given by
\be\label{4.31}
\hat{v}=\left\{
\begin{aligned}
&N^{0}_-v^{(1)}(N^{0}_+)^{-1},~~~ k\in\hat{\Gamma}\cap D_\varepsilon(0),\\
&(N^{0})^{-1},\qquad\qquad k\in\partial D_\varepsilon(0),\\
&v^{(1)},\qquad\qquad\quad~~ k\in\hat{\Gamma}\setminus\overline{D_\varepsilon(0)}.
\end{aligned}
\right.
\ee
\begin{lemma}\label{lemma4.3}
Let $\hat{\omega}=\hat{v}-I$. For each $1\leq n\leq\infty$, the following estimates hold:
\begin{align}
&\|\hat{\omega}\|_{L^n(\partial D_\varepsilon(0))}\leq Ct^{-\frac{1}{5}},\label{4.32}\\
&\|\hat{\omega}\|_{L^n(\Gamma^\varepsilon)}\leq Ct^{-\frac{1}{5}(1+\frac{1}{n})},\label{4.33}\\
&\|\hat{\omega}\|_{L^n(\bfR\setminus[-k_0,k_0])}\leq Ct^{-\frac{3}{2}},\label{4.34}\\
&\|\hat{\omega}\|_{L^n(\hat{\Gamma}')}\leq C\e^{-ct}.\label{4.35}
\end{align}
\end{lemma}
\begin{proof}
The estimate \eqref{4.32} follows from \eqref{4.20}. For $k\in\Gamma^\varepsilon$, we have
\berr
\hat{\omega}=N^{0}_-(v^{(1)}-v^0)(N^{0}_+)^{-1},
\eerr
as a consequence, \eqref{4.18} and \eqref{4.19} imply \eqref{4.33}. On $\bfR\setminus[-k_0,k_0]$, the jump matrix $\hat{\omega}$ only involves the small remainder $r_r$, so the estimate \eqref{4.34} holds as a consequence of Lemma \ref{lemma4.1} and \eqref{4.18}. Finally, \eqref{4.35} follows from $\e^{-t|\text{Re}\theta|}\leq C\e^{-ct}$ uniformly on $\hat{\Gamma}'$.
\end{proof}

As the discussion in Subsection 3.3, the estimates in Lemma \ref{lemma4.3} show that the RH problem \eqref{4.30} for $\hat{N}$ has a unique solution given by
\be\label{4.36}
\hat{N}(x,t;k)=I+\frac{1}{2\pi\ii}\int_{\hat{\Gamma}}(\hat{\mu}\hat{\omega})(x,t;s)\frac{\dd s}{s-k},
\ee
where $\hat{\mu}=I+(I-\hat{C}_{\hat{w}})^{-1}\hat{C}_{\hat{w}}I,$ and $\hat{C}_{\hat{w}}f=\hat{C}_-(f\hat{w})$, $\hat{C}$ denote the Cauchy operator associated with $\hat{\Gamma}$. Moreover, the function $\hat{\mu}(x,t;k)$ satisfies
\be\label{4.37}
\|\hat{\mu}(x,t;\cdot)-I\|_{L^2(\hat{\Sigma})}=O(t^{-\frac{1}{5}}),\quad t\rightarrow\infty,~(x,t)\in\mathcal{P}.
\ee
It follows from \eqref{4.36} that
\be\label{4.38}
\lim_{k\rightarrow\infty}k(\hat{N}(x,t;k)-I)=-\frac{1}{2\pi\ii}\int_{\hat{\Gamma}}(\hat{\mu}\hat{\omega})(x,t;k)\dd k.
\ee
By \eqref{4.21}, \eqref{4.32} and \eqref{4.37}, we can get
\begin{align}
&\quad-\frac{1}{2\pi\ii}\int_{\partial D_\varepsilon(0)}(\hat{\mu}\hat{\omega})(x,t;k)\dd k\nn\\
&=-\frac{1}{2\pi\ii}\int_{\partial D_\varepsilon(0)}\hat{\omega}(x,t;k)\dd k-\frac{1}{2\pi\ii}\int_{\partial D_\varepsilon(0)}(\hat{\mu}(x,t;k)-I)\hat{\omega}(x,t;k)\dd k\nn\\
&=-\frac{1}{2\pi\ii}\int_{\partial D_\varepsilon(0)}\bigg((N^{0})^{-1}(x,t;k)-I\bigg)\dd k+O(\|\hat{\mu}-I\|_{L^2(\partial D_\varepsilon(0))}\|\hat{w}\|_{L^2(\partial D_\varepsilon(0))})\nn\\
&=\frac{\ii N_1^0(y)}{(20\beta t)^{\frac{1}{5}}}+O(t^{-\frac{2}{5}}),\quad t\rightarrow\infty.\nn
\end{align}
Using \eqref{4.33} and \eqref{4.37}, we have
\bea
\int_{\Gamma^\varepsilon}(\hat{\mu}\hat{\omega})(x,t;k)\dd k&=&\int_{\Gamma^\varepsilon}\hat{\omega}(x,t;k)\dd k+\int_{\Gamma^\varepsilon}(\hat{\mu}(x,t;k)-I)\hat{\omega}(x,t;k)\dd k\nn\\
&\leq&\|\hat{\omega}\|_{L^1(\Gamma^\varepsilon)}+\|\hat{\mu}-I\|_{L^2(\Gamma^\varepsilon)}\|\hat{\omega}\|_{L^2(\Gamma^\varepsilon)}\nn\\
&\leq&Ct^{-\frac{2}{5}},\quad t\rightarrow\infty.\nn
\eea
By \eqref{3.34} and \eqref{3.37}, the contribution from $\bfR\setminus[-k_0,k_0]$ to the right-hand side of \eqref{4.38} is
$O(t^{-\frac{3}{2}})$, and similarly, by \eqref{3.35} and \eqref{3.37}, the contribution from $\hat{\Gamma}'$ to the right-hand side of \eqref{4.38} is $O(\e^{-ct}),$ as $t\rightarrow\infty.$
Thus, we obtain the following important relation
\bea\label{4.39}
\begin{aligned}
\lim_{k\rightarrow\infty}k(\hat{M}(x,t;k)-I)&=\frac{\ii N_1^0(y)}{(20\beta t)^{\frac{1}{5}}}+O(t^{-\frac{2}{5}}),\quad t\rightarrow\infty.
\end{aligned}
\eea
Recalling the definition of $N_1^0(y)$ in \eqref{4.22} and the relation \eqref{4.6}, we obtain our another main results stated in the Theorem \ref{the4.3}.
\begin{remark}
We did not directly consider the asymptotic behavior of the solution to equation \eqref{1.1} in region $\mathcal{P}$ because there is no suitable scale transformation to eliminate $t$ from the coefficients of $k^3$ and $k^5$ terms of the phase function $\Phi(k)$ given by \eqref{3.2} at the same time.
\end{remark}

{\bf Acknowledgments.}

N. Liu was supported by the China Postdoctoral Science Foundation under Grant no. 2019TQ0041.

\begin{center}\textbf{Appendix A. Fourth order Painlev\'e II RH problem}\end{center}
\begin{figure}[htbp]
\centering
\includegraphics[width=3.5in]{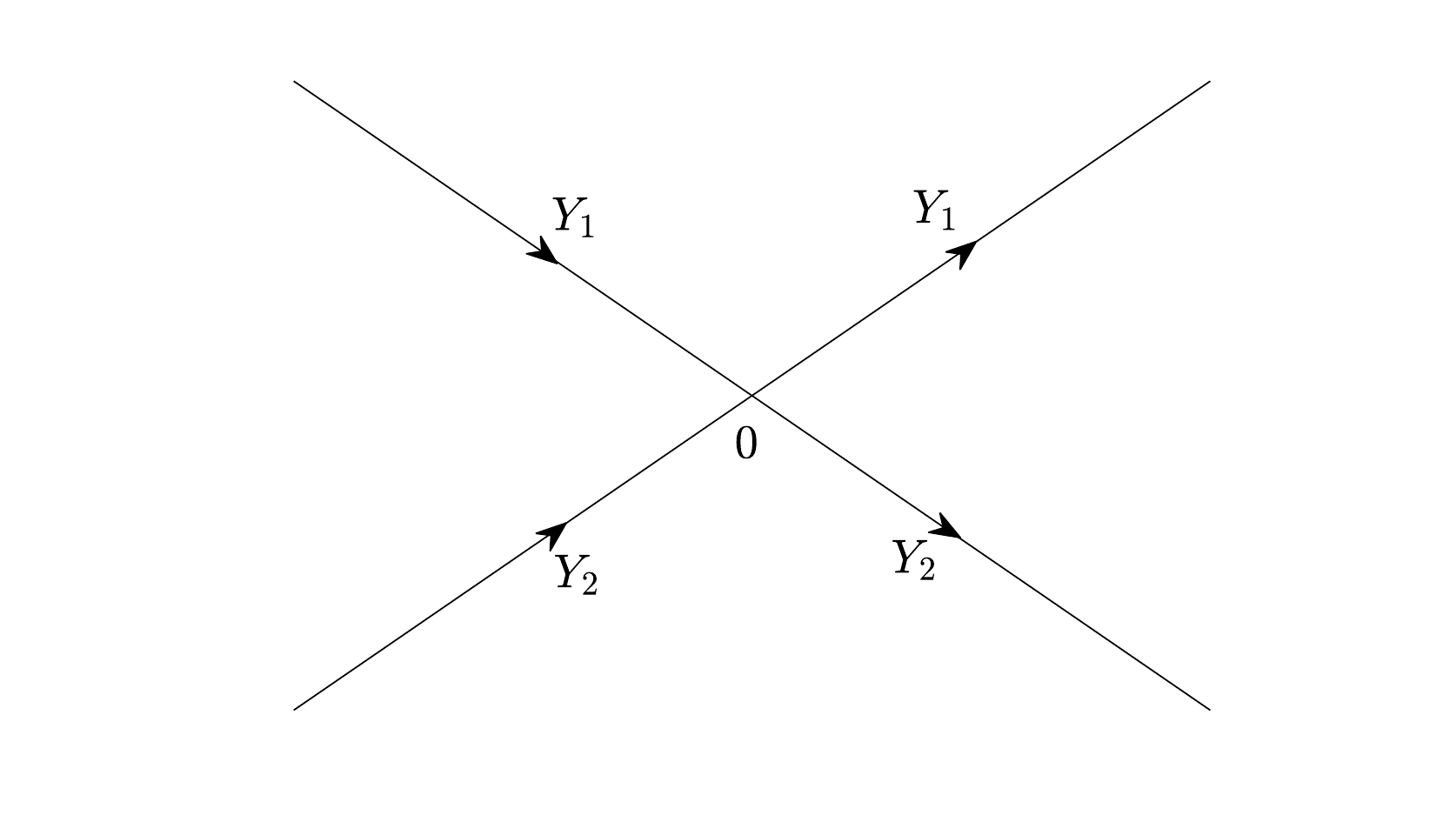}
\caption{The oriented contour $Y$.}\label{fig9}
\end{figure}
Let $Y$ denote the contour $Y=Y_1\cup Y_2$ oriented as in Fig. \ref{fig9}, where
\berr
Y_1=\{l\e^{\frac{\pi\ii}{6}}|l\geq0\}\cup\{l\e^{\frac{5\pi\ii}{6}}|l\geq0\},\quad
Y_2=\{l\e^{-\frac{\pi\ii}{6}}|l\geq0\}\cup\{l\e^{-\frac{5\pi\ii}{6}}|l\geq0\}.
\eerr
\textbf{Lemma A.1} (Fourth order Painlev\'e II RH problem)
\textit{Let $s\in\bfC$ be a complex number. Then the following RH problems parametrized by $y\in\bfR,s\in\bfC$:
\be\tag{A.1}
\left\{
\begin{aligned}
&N^Y_+(y,z)=N^Y_-(y,z)v^Y(y,z),~~z\in Y,\\
&N^Y(y,z)\rightarrow I,\qquad\qquad\quad\qquad~z\rightarrow\infty,
\end{aligned}
\right.
\ee
where the jump matrix $v^Y(y,z)$ is defined by
\be\tag{A.2}
v^Y(y,z)=\left\{
\begin{aligned}
&\begin{pmatrix}
1 &~ 0\\
s\e^{2\ii(\frac{4}{5}z^5+yz)} &~ 1
\end{pmatrix},\qquad z\in Y_1,\\
&\begin{pmatrix}
1 &~ \bar{s}\e^{-2\ii(\frac{4}{5}z^5+yz)}\\
0 &~ 1
\end{pmatrix},~~\quad z\in Y_2.
\end{aligned}
\right.
\ee
has a unique solution $N^Y(y,z)$ for each $y\in\bfR$. Moreover, there exists smooth functions $\{N_j^Y(y)\}^4$ of $y\in\bfR$ with decay as $y\rightarrow\infty$ such that
\be\tag{A.3}
N^Y(y,z)=I+\sum_{j=1}^4\frac{N_j^Y(y)}{z^j}+O(z^{-5}), \quad z\rightarrow\infty,
\ee
uniformly for $y$ in compact subsets of $\bfR$ and for $\arg z\in[0,2\pi]$. The leading coefficient $N_1^Y$ is given by
\be\tag{A.4}
N_1^Y(y)=\ii\begin{pmatrix}
-2\int^y_{\infty}u_p^2(y')\dd y' ~& u_p(y)\\
u_p(y) ~& 2\int^y_{\infty}u_p^2(y')\dd y'
\end{pmatrix},
\ee
where the real-valued function $u_p(y)$ satisfies the following fourth order Painlev\'e II equation (see \cite{KS})
\be\tag{A.5}
u_p^{''''}(y)+40u_p^2(y)u_p^{''}(y)+40u_p(y)u_p'^2(y)+96u_p^5(y)+4yu_p(y)=0.
\ee}
\begin{proof}
The jump matrix $v^Y$ admits the symmetries
\be\tag{A.6}
v^Y(y,z)=(v^Y)^\dag(y,\bar{z})=\overline{v^Y(y,-\bar{z})}.
\ee
We infer from the first of these symmetries that the RH problem for $N^Y(y,z)$ admits a vanishing lemma, as a consequence, there exists a unique solution $N^Y(y,z)$ which admits an expansion of the form (A.3). Assume that
\be\tag{A.7}
(N^Y)^{-1}(y,z)=I+\sum_{j=1}^4\frac{\varphi_j(y)}{z^j}+O(z^{-5}),
\ee
a direct calculation shows that
\begin{align}\tag{A.8}
\varphi_1&=-N_1^Y,~\varphi_2=(N_1^Y)^2-N_2^Y,~\varphi_3=N_1^YN_2^Y+N_2^YN_1^Y-(N_1^Y)^3-N_3^Y,\\
\varphi_4&=(N_1^Y)^4+N_1^YN_3^Y+N_3^YN_1^Y-(N_1^Y)^2N_2^Y-N_1^YN_2^YN_1^Y-N_2^Y(N_1^Y)^2+(N_2^Y)^2-N_4^Y.\nn
\end{align}
Let $\phi(y,z)=N^Y(y,z)\e^{-\ii(\frac{4}{5}z^5+yz)\sigma_3}$. Then the function $\mathcal{Y}(y,z)$ defined by
\be\tag{A.9}
\mathcal{Y}=\phi_y\phi^{-1}=(N_y^Y-\ii zN^Y\sigma_3)(N^Y)^{-1}
\ee
is an entire function of $z$, hence according to (A.3), (A.7) and (A.8), we have
\be\tag{A.10}
\mathcal{Y}(y,z)=-\ii z\sigma_3+\ii[\sigma_3,N_1^Y].
\ee
Thus, we find
\be\tag{A.11}
N_y^Y-\ii zN^Y\sigma_3=\mathcal{Y}N^Y.
\ee
Substituting the expansion (A.3) into (A.11) and collecting terms with $O(z^{-n})$, one can get
\be\tag{A.12}
\begin{aligned}
N_{1y}^Y+\ii[\sigma_3,N_2^Y]&=\ii[\sigma_3,N_1^Y]N_1^Y,\\
N_{2y}^Y+\ii[\sigma_3,N_3^Y]&=\ii[\sigma_3,N_1^Y]N_2^Y,\\
N_{3y}^Y+\ii[\sigma_3,N_4^Y]&=\ii[\sigma_3,N_1^Y]N_3^Y.
\end{aligned}
\ee
Accordingly, since
\be\tag{A.13}
\mathcal{Z}=\phi_z\phi^{-1}=\bigg(N_z^Y-\ii(4z^4+y)N^Y\sigma_3\bigg)(N^Y)^{-1}
\ee
is entire, and thus we get
\be\tag{A.14}
\mathcal{Z}=\mathcal{Z}_0+\mathcal{Z}_1z+\mathcal{Z}_2z^2+\mathcal{Z}_3z^3+\mathcal{Z}_4z^4.
\ee
Substituting the expansion (A.3) and (A.7) into (A.13), it follows from (A.8) and (A.12) that
\be\tag{A.15}
\begin{aligned}
&\mathcal{Z}_4=-4\ii\sigma_3,~\mathcal{Z}_3=4\ii[\sigma_3,N_1^Y],~\mathcal{Z}_2=-4N_{1y}^Y,
~\mathcal{Z}_1=-4N_{2y}^Y+4N_{1y}^YN_{1}^Y,\\
&\mathcal{Z}_0=-4N_{3y}^Y-4N_{1y}^Y(N_{1}^Y)^2+4N_{1y}^YN_{2}^Y+4N_{2y}^YN_{1}^Y-\ii y\sigma_3.
\end{aligned}\ee
We have shown that $\phi$ obeys the Lax pair equations
\be\tag{A.16}
\left\{
\begin{aligned}
\phi_y=&\mathcal{Y}\phi,\\
\phi_z=&\mathcal{Z}\phi,
\end{aligned}
\right.
\ee
where $\mathcal{Y}$ and $\mathcal{Z}$ are given by (A.10) and (A.14), respectively.

The symmetries (A.6) of the jump matrix $v^Y(y,z)$ implies that $N^Y(y,z)$ satisfies the symmetries
\be\tag{A.17}
N^Y(y,z)=(N^Y)^\dag(y,\bar{z})^{-1}=\sigma_2N^Y(y,-z)\sigma_2.
\ee
In particular, the coefficients $N_1^Y(y)$, $N_2^Y(y)$ and $N^Y_3(y)$ satisfy
\be\tag{A.18}
\begin{aligned}
N_1^Y=&-(N_1^Y)^\dag=-\sigma_2N_1^Y\sigma_2,\\
N_2^Y=&\sigma_2N_2^Y\sigma_2,~N_3^Y=-\sigma_2N_3^Y\sigma_2.
\end{aligned}
\ee
Therefore, we can write
\be\tag{A.19}
\begin{aligned}
N_1^Y(y)&=\begin{pmatrix}
\psi_1(y) ~& \psi_2(y)\\
\psi_2(y) ~& -\psi_1(y)
\end{pmatrix},\nn\\
N_2^Y(y)&=\begin{pmatrix}
f_1(y) ~& f_2(y)\\
-f_2(y) ~& f_1(y)
\end{pmatrix},\\
N_3^Y(y)&=\begin{pmatrix}
g_1(y) ~& g_2(y)\\
g_2(y) ~& -g_1(y)
\end{pmatrix},\nn
\end{aligned}
\ee
where $\{\psi_j(y),f_j(y),g_j(y)\}_1^2$ are complex-valued functions and $\psi_1(y),\psi_2(y)\in\ii\bfR$.
Then the compatibility condition
\be\tag{A.20}
\mathcal{Y}_z-\mathcal{Z}_y+\mathcal{Y}\mathcal{Z}-\mathcal{Z}\mathcal{Y}=0
\ee
of the Lax pair (A.16) can then rewrite as
\be\tag{A.21}
-\ii\sigma_3-\mathcal{Z}_{0y}+\ii[\sigma_3,N_1^Y]\mathcal{Z}_0-\ii\mathcal{Z}_0[\sigma_3,N_1^Y]=0,
\ee
since one can directly calculate that the coefficients of $z,z^2,z^3$ and $z^4$ in (A.20) vanish identically.
On the other hand, substituting (A.19) into (A.12), we find
\be\tag{A.22}
\left\{
\begin{aligned}
\psi_1'&=2\ii\psi_2^2,\\
\psi_2'+2\ii f_2&=-2\ii\psi_1\psi_2,\\
f_1'&=-2\ii\psi_2f_2,\\
f_2'+2\ii g_2&=2\ii\psi_2f_1,\\
g_1'&=2\ii\psi_2g_2.
\end{aligned}
\right.
\ee
Substituting (A.19) into (A.21) and using above relations, it follows from $(12)$-entry of (A.21) that
\be\tag{A.23}
\psi_2''''-40\psi_2^2\psi_2''-40\psi_2\psi_2'^2+96\psi_2^5+4y\psi_2=0,
\ee
however, the $(11)$-entry of (A.21) vanish identically. If we set $\psi_2(y)=\ii u_p(y)$, then $u_p(y)$ satisfies the fourth order Painlev\'e II equation (A.5). The lemma follows.
\end{proof}
\begin{center}\textbf{Appendix B. Model RH problem for sector $\mathcal{P}$}\end{center}
\begin{figure}[htbp]
\centering
\includegraphics[width=3.5in]{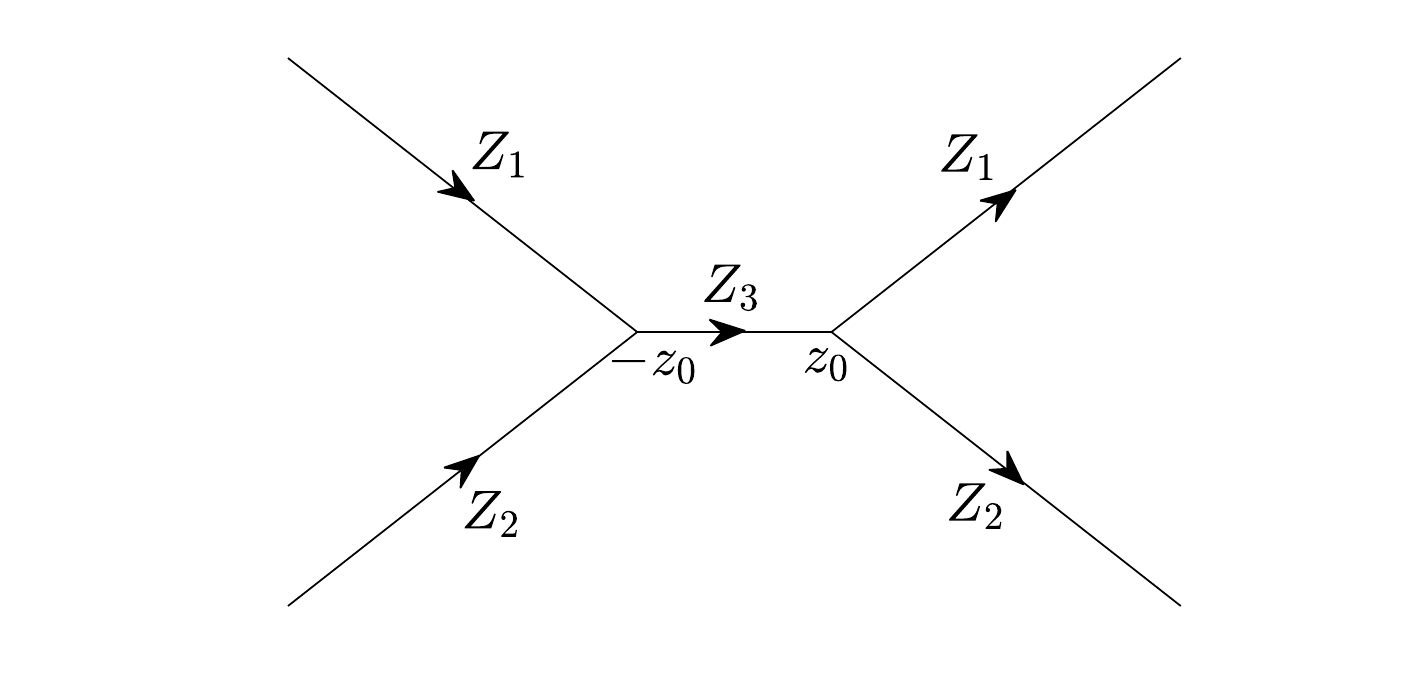}
\caption{The oriented contour $Z$.}\label{fig10}
\end{figure}
Given a number $z_0\geq0$, let $Z$ denote the contour $Z=Z_1\cup Z_1\cup Z_3$, where the line segments
\be\tag{B.1}
\begin{aligned}
Z_1&=\{z_0+l\e^{\frac{\pi\ii}{6}}|l\geq0\}\cup\{-z_0+l\e^{\frac{5\pi\ii}{6}}|l\geq0\},\\
Z_2&=\{z_0+l\e^{-\frac{\pi\ii}{6}}|l\geq0\}\cup\{-z_0+l\e^{-\frac{5\pi\ii}{6}}|l\geq0\},\\
Z_3&=\{l|-z_0\leq l\leq z_0\}
\end{aligned}
\ee
are oriented as in Fig. \ref{fig10}. It turns out that the long time asymptotics in the sector $\mathcal{P}$ is related to the solution $N^Z$ of the following family of RH problems parametrized by $y<0,s\in\bfC,z_0\geq0$:
\be\tag{B.2}
\left\{
\begin{aligned}
&N^Z_+(s,y,z,z_0)=N^Z_-(s,y,z,z_0)v^Z(s,y,z,z_0),~z\in Z,\\
&N^Z(s,y,z,z_0)\rightarrow I,\qquad\qquad\qquad\qquad\quad\qquad~z\rightarrow\infty,
\end{aligned}
\right.
\ee
where the jump matrix $v^Z(s,y,z,z_0)$ is defined by
\be\tag{B.3}
v^Z(s,y,z,z_0)=\left\{
\begin{aligned}
&\begin{pmatrix}
1 &~ 0\\
s\e^{2\ii(\frac{4}{5}z^5+yz)} &~ 1
\end{pmatrix},\qquad\qquad\qquad\qquad~~ z\in Z_1,\\
&\begin{pmatrix}
1 &~ \bar{s}\e^{-2\ii(\frac{4}{5}z^5+yz)}\\
0 &~ 1
\end{pmatrix},~~\quad\qquad\qquad\qquad~~ z\in Z_2,\\
&\begin{pmatrix}
1 &~ \bar{s}\e^{-2\ii(\frac{4}{5}z^5+yz)}\\
0 &~ 1
\end{pmatrix}\begin{pmatrix}
1 &~ 0\\
s\e^{2\ii(\frac{4}{5}z^5+yz)} &~ 1
\end{pmatrix},\quad z\in Z_3.
\end{aligned}
\right.
\ee
\textbf{Lemma B.1} (Model RH problem for sector $\mathcal{P}$) \textit{Define the parameter subset
\be\tag{B.4}
\Bbb P=\{(y,t,z_0)\in\bfR^3|-C_1\leq y<0,t\geq3,\sqrt[4]{-y}/\sqrt{2}\leq z_0\leq C_2\},
\ee
where $C_1,C_2>0$ are constants. Then for $(y,t,z_0)\in\Bbb P$, the RH problem (B.2) has a unique solution $N^Z(s,y,z,z_0)$ which satisfies
\be\tag{B.5}
N^Z(s,y,z,z_0)=I+\frac{\ii}{z}\begin{pmatrix}
-2\int^y_{\infty}u_p^2(y')\dd y' ~& u_p(y)\\
u_p(y) ~& 2\int^y_{\infty}u_p^2(y')\dd y'
\end{pmatrix}+O\bigg(\frac{1}{z^2}\bigg),\quad z\rightarrow\infty,
\ee
where $u_p(y)$ denotes the solution of the fourth order Painlev\'e II equation (A.5) and $N^Z(s,y,z,z_0)$ is uniformly bounded for $z\in\bfC\setminus Z$. Furthermore, $N^Z$ obeys the symmetries
\be\tag{B.6}
N^Z(s,y,z,z_0)=(N^Z)^\dag(s,y,\bar{z},z_0)^{-1}=\sigma_2N^Z(s,y,-z,z_0)\sigma_2.
\ee
}
\begin{proof}
Note that
\berr
\text{Re}\bigg(2\ii\bigg(\frac{4}{5}z^5+yz\bigg)\bigg)\leq l^2\bigg(-\frac{4}{5}l^3-4\sqrt{3}z_0l^2-16z_0^2l-8\sqrt{3}z_0^3\bigg)
\eerr
for all $z=z_0+l\e^{\frac{\pi\ii}{6}}$ and $z=-z_0+l\e^{\frac{5\pi\ii}{6}}$ with $l\geq0,z_0\geq0$ and $-4z_0^4\leq y<0$. Thus, we have
\berr
|\e^{2\ii(\frac{4}{5}z^5+yz)}|\leq C\e^{-|z\pm z_0|^2(\frac{4}{5}|z\pm z_0|^3+4\sqrt{3}z_0|z\pm z_0|^2+16z_0^2|z\pm z_0|+8\sqrt{3}z_0^3)},\quad z\in Z_1.
\eerr
Analogous estimates hold for $z\in Z_2$. However, $|\e^{\pm2\ii(\frac{4}{5}z^5+yz)}|=1$ for $z\in Z_3$, this shows that $v^Z-I$ exponentially fast as $z\rightarrow\infty$.

The jump matrix $v^Z$ obeys the same symmetries (A.6) as $v^Y$. In particular, $v^Z$ is Hermitian and positive definite on $Z\cap\bfR$ and satisfies $v^Z(s,y,z,z_0)=(v^Z)^\dag(s,y,z,z_0)$ on $Z\setminus\bfR$. This implies the existence of a vanishing lemma from which we deduce the unique existence of the solution $N^Z$. The symmetries (B.6) follow from the symmetries of $v^Z$. Moreover, the RH problem (B.2) for $N^Z(s,y,z,z_0)$ can be transformed into the RH problem (A.1) for $N^Y(y,z)$ up
to a trivial contour deformation. Thus (B.5) follows from (A.3) and (A.4).
\end{proof}

\medskip
\small{

}
\end{document}